\documentclass[11pt]{amsart}
\usepackage{amsmath, amsfonts, amsthm, amstext, amssymb}
\usepackage{tikz-cd} 
 \usepackage[T1]{fontenc}

\theoremstyle{plain} 
\newtheorem{thm}{Theorem}[section] 
\newtheorem{thmy}{Theorem}

\newenvironment{thmx}{\stepcounter{thm}\begin{thmy}}{\end{thmy}}
\newtheorem*{thm*}{Theorem}
\newtheorem{prop}[thm]{Proposition}
\newtheorem{cor}[thm]{Corollary}

\newtheorem{lem}[thm]{Lemma}
\newtheorem*{prop*}{Proposition} 

\newtheorem*{conj*}{Conjecture}
\newtheorem{meta-conj}[thm]{Meta-Conjecture}
\newtheorem*{meta-conj*}{Meta-Conjecture}

\theoremstyle{definition}
\newtheorem{defin}[thm]{Definition}

\newtheorem{conv}[thm]{Convention}
\newtheorem{notation}[thm]{Notation}
\newtheorem{rem}[thm]{Remark} 
\newtheorem{exm}[thm]{Example}
\newtheorem{const}[thm]{Construction}

\newcommand{\bxi}{\bar{\xi}}

\newcommand{\f}{\mathbb{F}}
\newcommand{\z}{\mathbb{Z}}

\newcommand{\ca}{\mathcal{A}}
\newcommand{\ce}{\mathcal{E}}

\def\co{\colon\thinspace}

\DeclareMathOperator{\bT}{\mathbb{T}}

\DeclareMathOperator{\id}{\rm id}
\DeclareMathOperator{\bfilt}{\rm bfilt}

\DeclareMathOperator{\can}{\rm can}

\DeclareMathOperator{\im}{\rm im}
\DeclareMathOperator{\Ext}{\rm Ext}
\DeclareMathOperator{\Cotor}{\rm Cotor}
\DeclareMathOperator{\THH}{THH}
\DeclareMathOperator{\HH}{HH}
\DeclareMathOperator{\TC}{TC}
\DeclareMathOperator{\TP}{TP}
\DeclareMathOperator{\K}{K}

\DeclareMathOperator{\sk}{sk}
\DeclareMathOperator{\Gr}{Gr}

\DeclareMathOperator{\Th}{Th}
\DeclareMathOperator{\GL}{GL}
\DeclareMathOperator{\colim}{colim}
\DeclareMathOperator{\holim}{holim}
\DeclareMathOperator{\hocolim}{hocolim}
\DeclareMathOperator{\Hom}{Hom}

\usepackage[colorlinks=true, linkcolor=blue, menucolor=blue]{hyperref}
\usepackage[shortlabels]{enumitem}

\title{A homological approach to chromatic complexity of algebraic K-theory}  

\author{Gabriel Angelini-Knoll}
\address{Universit{\'e} Paris 13, LAGA, CNRS, UMR 7539, F-93430, Villetaneuse, France}
\email{angelini-knoll@math.univ-paris13.fr}

\author{J.D. Quigley}
\address{Department of Mathematics, University of Virginia, Charlottesville, Virginia, U.S.A.}
\email{mbp6pj@virginia.edu}

\begin{document}

\begin{abstract}
The family of Thom spectra~$y(n)$ interpolates between the sphere spectrum and the mod two Eilenberg--MacLane spectrum. Computations of Mahowald, Ravenel, Shick, and the authors show that the associative ring spectrum~$y(n)$ has type~$n$. Using trace methods, we give evidence that algebraic K-theory preserves this chromatic complexity. Our approach sheds light on the chromatic complexity of topological negative cyclic homology and topological periodic cyclic homology, which approximate algebraic K-theory and are of independent interest. Our main contribution is a homological approach that can be applied in great generality, such as to associative ring spectra~$R$ without additional structure whose coefficient rings are not completely understood. 
\end{abstract}

\maketitle

\tableofcontents

\section{Introduction}
In~\cite{AR08}, Ausoni--Rognes laid out an ambitious collection of conjectures about how the arithmetic of associative ring spectra can be understood using telescopically localized algebraic K-theory.\footnote{In this introduction and the abstract, we refer to $\mathbb{E}_{1}$~ring spectra as associative ring spectra and $\mathbb{E}_{\infty}$~ring spectra as commutative ring spectra.}~One of the essential features of these conjectures is that algebraic K-theory should increase chromatic complexity by one. For example, we say a spectrum~$X$ has height~$n$ if $K(n)_{*}X\ne 0$ and~$K(n+k)_{*}X=0$ for $k>0$, where $K(i)$ is the~$i$-th Morava K-theory at a fixed prime~$p$. It is now known by celebrated work of Burklund--Schlank--Yuan~\cite{BSY22} that for commutative ring spectra algebraic K-theory increases height by exactly one.
 
In the case of more general associative ring spectra, additional subtleties arise. For example, if $R$ is an associative ring spectrum, but not a more structured ring spectrum, then the algebraic K-theory of $R$ is not a ring spectrum. More generally, all of the invariants that are commonly used to compute algebraic K-theory through trace methods, such as topological Hochschild homology ($\THH$), topological negative cyclic homology ($\TC^{-}$), topological periodic cyclic homology ($\TP$), and topological cyclic homology ($\TC$), are not rings when applied to associative ring spectra without additional structure. Additionally, when the homotopy ring~$\pi_{*}R$ of an associative ring spectrum is not completely understood, this further complicates the study of its algebraic K-theory. In this paper, we demonstrate a homological approach to understanding the chromatic complexity of algebraic K-theory, building on ideas of Bruner--Rognes~\cite{BR05} and Lun{\o}e-Nielsen--Rognes~\cite{LNR11}, which can overcome these hurdles. 

We consider a family of $2$-primary associative ring spectra 
\[ S=y(0)\to y(1)\to \dots \to y(\infty)=H\f_2 \]
originally defined by Mahowald~\cite{Mah79}. Proposition \ref{cczn} (cf.~\cite{MRS01} for odd primes) implies that 
\[ n = \operatorname{min}\{ m : K(m)_*(y(n))\ne 0\},\]
so $y(n)$ may be considered type~$n$ (even though it is not a finite spectrum). 

We are interested in understanding $K(m)_*\K(y(n))$, but our main theorem is about an approximation to it. The Dundas--Goodwillie--McCarthy theorem~\cite{DGM12} and a result of Nikolaus--Scholze~\cite{NS18} reduce us to studying topological periodic and negative cyclic homology, $\TP(y(n))$ and~$\TC^-(y(n))$, instead of algebraic K-theory~$\K(y(n))$. Specifically, by~\cite[Theorem~7.3.1.8]{DGM12} there is a fiber sequence
\[ \K(y(n))_{2}\overset{tr}{\longrightarrow} \TC(y(n))_{2} \longrightarrow  \Sigma^{-1}H\mathbb{Z}_{2} \]
and by~\cite[Corollary~1.5]{NS18} there is a fiber sequence
\[ \TC(y(n))_{2} \longrightarrow \TC^{-}(y(n))_{2}\overset{\can-\varphi}{\longrightarrow}  \TP(y(n))_{2} \,.\]
If  $\TP(y(n))$ and~$\TC^-(y(n))$ are $K(m)$-acyclic, then the algebraic K-theory~$\K(y(n))$ is also $K(m)$-acyclic. On the other hand, to show that (periodic) Morava K-theory $K(m)_*X$ vanishes, it suffices to show that connective Morava K-theory $k(m)_*X$ is $v_m$-torsion. This reduces us to understanding $v_m$-torsion in $k(m)_*\TP(y(n))$ and~$k(m)_*\TC^-(y(n))$. 

Recall that for any associative ring spectrum $R$, $\TP(R)$ (resp. $\TC^-(R)$) is the inverse limit of its Greenlees (resp. skeletal) filtration~\cite{GM95}. If $E$ is a generalized homology theory, the \emph{continuous $E$-homology} $E^c_*\TP(R)$ is obtained by taking the inverse limit of the $E$-homology of this filtration. Note that in full generality, continuous $E$-homology does not agree with $E$-homology, but in some cases this difficulty can be overcome (compare with the work of the first author and Salch~\cite{AKS20}). 

Our main theorem is the following:
\begin{thmx}[{Theorem \ref{thm:continuous-Morava-k-theory-vanishing}}]\label{main theorem}
The $k(m)_{*}$-module $k(m)_{*}^{c}(\TP(y(n)))$ is simple $v_{m}$-torsion for all $0<m\le n$. Moreover, the $k(m)_{*}$-module $k(m)_{*}^{c}(\TC^-(y(n)))$ is simple $v_{m}$-torsion for each $0<m<n$. 
\end{thmx}
 
Analogous to~\cite{BR05}, one might hope to approach $\pi_{*}\TP(y(n))_{2}^{\wedge}$ and $\pi_{*}\TC^{-}(y(n))_{2}^{\wedge}$ using the inverse-limit $k(n)$-based Adams spectral sequences, constructed as in~\cite{LNR12}. However, usually it is difficult to identify the $\mathrm{E}_{2}$-page of the $k(n)$-based Adams spectral sequence. A consequence of our main theorem and work of \cite[Theorem~5.4]{BP25} is that the $\mathrm{E}_{2}$-page of this spectral sequence is computable in terms of homological algebra of quiver representations.  

As stated before, this paper overcomes two technical hurdles: working with ring spectra whose homotopy rings are not understood and working with associative ring spectra. The first hurdle is overcome by working with homology and Margolis homology in order to understand $k(n)$-homology. Much of our technical work involves computing the homology of the topological Hochschild homology of $y(n)$, the continuous homology of $\TC^{-}$ and $\TP$ of $y(n)$, and associated  Margolis homology groups. Our work suggests for example that $\mathrm{TP}$ of $y(n)$ has an interpretation in terms of the $\bT$-Tate construction of a Thom spectrum of higher height, $z(n)/v_{n}$, which we construct during our analysis. Here $\bT\subset \mathbb{C}$ denotes the circle regarded as a topological group with the usual topology. 

In order to execute these computations, we must overcome the second hurdle of working with an associative ring spectrum. This is overcome by a careful understanding of the map $\THH(y(n))\longrightarrow \THH(\mathbb{F}_{2})$ to an $\mathbb{E}_\infty$ ring spectrum. After this paper appeared in preprint form, similar techniques have been employed in work of the first author, Hahn, and Wilson in order to study the algebraic K-theory of Morava K-theory~\cite{AKHW24}. As part of our technical tool set, we also produce a new inverse-limit May--Ravenel spectral sequence, which may be of independent interest. 

\begin{rem}
Unfortunately, Theorem~\ref{main theorem} does not directly imply that $K(m)_{*}\K(y(n))=0$ for $0<m<n$ as desired. See Remark~\ref{rem:continuos-Morava-K-theory-to-Morava-K-theory} for more details about the relationship between the two. Nevertheless, work of Land--Mathew--Meier--Tamme~\cite{LMMT24} proves directly that $K(m)_{*}\K(y(n))=0$ for $0<m<n$, without appealing to trace methods. The approach in their paper is entirely complementary to the approach in our work and we present these results as an alternative that can also shed light on the chromatic complexity of topological periodic cyclic homology, which is closely related to prismatic cohomology~\cite{BMS18,BL22,HRW22}. For example, these methods are used in~\cite{AKS20} to prove $K(m)_{*}F(\mathrm{BP}\langle n\rangle)=0$ for $F\in \{\mathrm{K},\mathrm{TC}^{-},\mathrm{TP}\}$ and $m\ge n+2$ under some conditions on $\mathrm{BP}\langle n\rangle$. We also point out the relevant work of Keenan--McCandless~\cite{KM23} on $K(m)$-acyclicity of topological restriction homology $\mathrm{TR}$. 
\end{rem} 

\begin{rem}
As mentioned above, the idea of ``homological trace methods" was introduced in Bruner--Rognes~\cite{BR05} and studied by Lun{\o}e-Nielsen~\cite{LN05} and Lun{\o}e-Nielsen--Rognes~\cite{LNR11,LNR12}. What distinguishes our results from the results in op.~cit. is the absence of multiplicative structure on $\THH(y(n))$, as well as the additional step of passing from continuous mod $p$ homology to continuous connective Morava K-theory. We also note that the use of homology to study algebraic K-theory predates trace methods altogether, such as in the seminal work of Quillen~\cite{Qui72} and Charney~\cite{Cha80}. 
\end{rem}

\subsection{Outline}

In Section~\ref{Sec:Thom}, we recall the construction and basic properties of the spectra~$y(n)$. 
We show the vanishing of certain Morava K-theory of $y(n)$ using Margolis homology and the localized Adams spectral sequence. We also construct Thom spectra~$z(n)$ which are integral analogs of $y(n)$, i.e. they are spectra which interpolate between the sphere spectrum~$\mathbb{S}$ and the integral Eilenberg--MacLane spectrum~$H\z$. We show that the spectra~$z(n)$ have a self-map~$v_n$ and, again using Margolis homology, we compute vanishing of Morava K-theory of $z(n)$ and the cofiber~$z(n)/v_n$ of this self-map. We believe that the spectra~$z(n)$ are of independent interest, for example see~\cite{Dev24,LMMT24}. 

In Section~\ref{Sec:THH}, we analyze the B{\"o}kstedt spectral sequence converging to the mod two homology of $\THH(y(n))$. We also prove a key technical proposition (Proposition~\ref{Prop:ymapF}) about the map~$H_*(\THH(y(n)))\to H_*(\THH(H\mathbb{F}_2))$ which we use in subsequent sections. 

In Section~\ref{Sec:TP}, we analyze the topological periodic cyclic homology $\TP(y(n)) := \THH(y(n))^{t\bT}$. 
The key tool is the \emph{homological} Tate spectral sequence constructed by Bruner--Rognes~\cite{BR05}. In Section~\ref{Sec:TCm}, we carry out a similar analysis for topological negative cyclic homology $\TC^-(y(n)) := \THH(y(n))^{h\bT}$. 

In Section~\ref{Sec:MR}, we prove the main theorem. Our proof uses a new spectral sequence, the \emph{inverse limit May--Ravenel spectral sequence}, to resolve hidden extensions which arise in the study of continuous homology using the homological Tate spectral sequence and the homological homotopy fixed point spectral sequence (cf.~Remark~\ref{hidden-extensions}). 

\subsection{Conventions}\label{conventions}
We fix $p=2$ throughout. We write $H_*(X)$ (resp. $H^*(X)$) for homology (resp. cohomology) of a space or spectrum~$X$ with coefficients in $\mathbb{F}_2$. 

We write $\ca:=H^*(H\f_2)$ for the ($2$-primary) Steenrod algebra, which is a Hopf algebra with generators~$Sq^{2^i}$ in degree $2i$ and relations given by the Adem relations. The dual of the Steenrod algebra will be denoted $\ca^{\vee} := H_*(H\f_2)$ and it is isomorphic to $P(\bxi_i \mid i\ge 1)$ where $\bxi_{i} := \chi(\xi_i)$  in degree $2^i-1$ is the image of the usual Milnor generators under the antipode~$\chi$ of the Hopf algebra~$\ca^{\vee}$. The coproduct~$\psi \co \ca^{\vee}\to \ca^{\vee}\otimes \ca^{\vee}$ is given by the formula
\begin{equation}\label{coprod dual steenrod} 
\psi(\bxi_k) = \sum_{i+j=k} \bxi_i\otimes \bxi_j^{2^i} \,.
\end{equation}
Let $E(n) := E(Q_0,\ldots,Q_n)$ denote the subalgebra of the Steenrod algebra generated by the first $n+1$ Milnor primitives, where $Q_i$ is in degree $2^i-1$. We let $\ce := E(Q_i : i\ge 0)$
denote the subalgebra generated by all of the Milnor primitives. We write $E(n)^{\vee}$ and $\ce^{\vee}$ for the $\f_2$-linear duals of these subalgebras, respectively. 

When referring to modules over a graded polynomial ring~$P(x)$ with $|x|=2k$  for some integer~$k$, we say that an element in a $P(x)$-module $m\in M$ is \emph{simple~$x$-torsion} if $x\cdot m=0$. We will say that a $P(x)$-module~$M$ is simple~$x$-torsion if every element in $M$ is simple~$x$-torsion. 

We use lowercase letters, e.g., `$\mathrm{fil}$' when discussing filtered spectra, and capitalize the first letter, e.g., `$\mathrm{Fil}$' or `$\mathrm{Gr}$', when discussing the associated filtration or associated graded on homology. 

\subsection{Acknowledgements}
The authors thank Vigleik Angeltveit, Mark Behrens, Teena Gerhardt, Mike Hill, John Rognes, Andrew Salch, and Sean Tilson for helpful discussions. The authors further thank John Rognes and anonymous referees for helpful comments on previous versions. The second author was partially supported by NSF grants DMS-1547292, DMS-2039316, and DMS-2414922, an AMS-Simons Travel Grant, and the Max Planck Institute for Mathematics in Bonn. This project has received funding from
the European Union's Horizon 2020 research and innovation programme under the Marie Sk\l{}odowska-Curie grant agreement No 1010342555. \thinspace \includegraphics[scale=0.1]{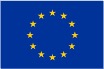}

\section{Families of Thom spectra}\label{Sec:Thom}
In this section, we recall the Thom spectra~$y(n)$ which interpolate between the sphere spectrum and the mod two Eilenberg--MacLane spectrum. We also introduce a family of spectra~$z(n)$ which interpolate between the sphere spectrum and the integral Eilenberg--MacLane spectrum. Basic properties of both families, such as their homology and multiplicative structure, are discussed in Section~\ref{Sec:ynznConstruction}. In Section~\ref{Sec:LocASS}, we recall Margolis homology and the localized Adams spectral sequence, and in Section~\ref{Sec:ynznCC}, we apply them to compute the chromatic complexity of the spectra $y(n)$, $z(n)$, and~$z(n)/v_n$ (Proposition~\ref{cczn}). Though some of the arguments in this section have been improved, the results remain essentially unchanged from the first version of this paper to appear in pre-print form. 

\subsection{Construction of the Thom spectra $y(n)$ and~$z(n)$}\label{Sec:ynznConstruction}
We begin with Mahowald's construction of $H\f_2$ and~$H\z$ as Thom spectra~\cite{Mah79}. Let $f = \Omega^2 w \co \Omega^2 S^3 \to \Omega^2 B^3O \simeq BO$ be the two-fold looping of the generator~$w\co S^3 \to B^3 O$ of $\pi_3(B^3O)\cong \pi_0(O) \cong \z/2$. Recall that for an $\mathbb{E}_{\infty}$~ring spectrum~$R$, one can construct the group-like $\mathbb{E}_\infty$~space $\GL_1R$~\cite{ABGHR14,MQRT77}. When $R=\mathbb{S}$ is the sphere spectrum, its delooping~$B\GL_1\mathbb{S}$ is a model for the classifying space of stable spherical fibrations. 
The classical $J$ homomorphism then gives a map of group-like $\mathbb{E}_\infty$~spaces $J\co O\to \GL_1\mathbb{S}$.
In~\cite[Sec.~2.6]{Mah79}, Mahowald showed that 
\begin{equation}\label{Hf2} 
H\f_2 \simeq \Th \left (\Omega^2 S^3 \overset{f}{\longrightarrow} BO\overset{BJ}{\longrightarrow}B\GL_1\mathbb{S} \right )
\end{equation}
where $\Th(-)$ is the Thom spectrum construction. We refer the reader to~\cite{ABG18} for a modern treatment of the Thom spectrum construction. 

Similarly, Mahowald~\cite[Prop.~2.8]{Mah79} proved that 
\[ H\z \simeq \Th \left (\Omega^2(S^3\langle3\rangle)\to \Omega^2 S^3 \overset{f}{\longrightarrow} BO \overset{BJ}{\longrightarrow} B\GL_1\mathbb{S} \right ) \]
where $S^3\langle3\rangle$ is the fiber of the map~$S^3 \to K(\z,3)$ and~$\iota\co S^3\langle 3\rangle\to S^3$ is the inclusion of the fiber. 

We now produce the spectra~$y(n)$ following~\cite[Sec.~4.5]{Mah79}. 
The James splitting gives an equivalence~$\Omega \Sigma S^2 \simeq J_\infty S^2$ where $J_\infty X$ is the James construction of the space $X$~\cite{Jam55}, so we can rewrite \eqref{Hf2} as $H\f_2 \simeq \Th(\Omega J_\infty S^2 \to B\GL_1\mathbb{S})$.
By truncating the James construction, one can define spectra $J_k S^2$, and there is an obvious inclusion $i_k\co J_k S^2 \hookrightarrow J_\infty S^2$. Taking $k=2^n-1$, one defines
\[ 
y(n) :=\Th \left (\Omega J_{2^n-1} S^2 \overset{f_n }{\longrightarrow}B\GL_1\mathbb{S} \right )
\]
where $f_n=BJ\circ f\circ \Omega i_{2^n-1}$. (Note that one needs to $p$-localize in order to construct $y(n)$ at odd primes, but this is not necessary at the prime $2$.) 

The fiber sequence $J_{2^{n}-1}S^2\to \Omega S^3 \to \Omega S^{2^{n+1}+1}$ implies that the map~$J_{2^{n}-1}S^2\to \Omega S^3$ is $(2^{n+1}-1)$-connected, cf.~\cite{Jam56}. Thus there is a map~$J_{2^n-1}S^2\to K(\z,2)$ given by truncating homotopy groups which is compatible with the map~$J_\infty S^2\to K(\z,2)$. 

\begin{const}\label{const 1}
Let $n \geq 1$. Write $J_{2^n-1}S^2\langle 2 \rangle$ for the fiber of the map~$J_{2^n-1} S^2 \to K(\mathbb{Z},2)$ given by truncating homotopy groups. Define 
\[
z(n) := \Th(\Omega (J_{2^n-1}S^2\langle 2 \rangle) \overset{g_n}{\longrightarrow} B\GL_1\mathbb{S} )
\]
where $g_n=BJ\circ f \circ \Omega i_{2^n-1}\circ \Omega \iota_{2^n-1}$ with $ \iota_k\co J_{k}S^2\langle 2 \rangle \to J_{k}S^2$ is the inclusion of the fiber.
There is a commutative diagram 
\[
\begin{tikzcd}
	\Omega (J_{2^n-1}S^2\langle 2  \rangle ) \arrow{r}{\Omega j_{2^n-1}}  \arrow{d}{\Omega \iota_{2^n-1}} & \Omega^2 S^3\langle3\rangle \arrow{d}{\iota} \arrow{dr}{g} && \\
	\Omega (J_{2^n-1}S^2) \arrow{r}{\Omega i_{2^n-1}} \arrow{d} & \Omega^2 S^3  \arrow{r}{f} \ar[d] & BO \arrow{r}{BJ} &  B\GL_1\mathbb{S}\\
	S^1 \arrow{r}{=} & S^1  & & 
\end{tikzcd}
\] 
where the left two columns are fiber sequences, all the maps in the upper right triangle are $2$-fold loop maps, and all the maps in the upper left square are $1$-fold loop maps. 
\end{const}

After our paper appeared in preprint form, this family of $\mathbb{E}_{1}$ ring spectra has also been considered in work of Devalapurkar~\cite{Dev24} and Land--Meier--Mathew--Tamme~\cite{LMMT24}.

\begin{lem}\label{cor z(n) E1}
The spectra~$z(n)$ are $\mathbb{E}_1$ ring spectra and the diagram 
\[
	\begin{tikzcd}
		z(n)\arrow{rr}{\mathrm{Th}(\Omega j_{2^n-1})} \arrow{d}[swap]{\mathrm{Th}(\Omega \iota_{2^n-1})} && H\mathbb{Z} \arrow{d}{\mathrm{Th}(\iota)} \\
		y(n) \arrow{rr}{\mathrm{Th}(\Omega i_{2^n-1})} && H\mathbb{F}_2 	
	\end{tikzcd}	
\]
is a commutative diagram in the category of $\mathbb{E}_1$ ring spectra for $n\ge 1$.
\end{lem}
\begin{proof}
This immediately follows from Lewis's theorem~\cite[Ch. IX]{LMS06} and Construction~\ref{const 1}. 
\end{proof}

The homology of the spectra~$y(n)$ interpolate between the homology of the sphere spectrum~$\mathbb{S} \simeq y(0)$ and the homology of the mod $2$ Eilenberg--MacLane spectrum~$H\f_2 \simeq y(\infty)$. Similarly, the homology of the spectra~$z(n)$ interpolate between the homology of $z(1)$ and the homology of the integral Eilenberg--MacLane spectrum~$z(\infty)=H\z$. More generally, we have the following ladder of interpolations between the sphere spectrum and the mod $2$ and integral Eilenberg--MacLane spectra:
\[
\begin{tikzcd}
& z(1) \arrow{r} \arrow{d} &  z(2) \arrow{r} \arrow{d}&  \cdots \arrow{r} & z(\infty) = H\z\arrow{d}\\
\mathbb{S}=y(0) \arrow{r} \arrow{ur} &y(1) \arrow{r}& y(2) \arrow{r} &\cdots \arrow{r} &y(\infty) = H\f_2  .
\end{tikzcd}
\]
Note that the structure of $\ca^{\vee}$ as an $\ca^{\vee}$-comodule is given by the coproduct~$\psi$, and more generally, the coaction on any sub-Hopf algebra of $\ca^{\vee}$ is defined to be the restriction of the coproduct. 
\begin{lem}\label{lem:homology}
The maps $\mathrm{Th}(\Omega i_{2^n-1})$ and~$\mathrm{Th}(\Omega j_{2^n-1})$ from Construction~\ref{const 1} induce isomorphisms onto their image in $H_*(H\f_2)$ and~$H_*(H\z)$ respectively,
\begin{align*}
H_*(y(n)) &\cong P(\bxi_1,\bxi_2,\ldots,\bxi_n), \\
H_*(z(n)) &\cong P(\bxi^2_1,\bxi_2,\ldots,\bxi_n),
\end{align*}
for each $1 \le n\le \infty$, where $|\bxi_i| = 2^i-1$ and $|\bxi_1^2| = 2$. The map~$\Th(\Omega \iota_{2^n-1} )\colon \thinspace z(n) \to y(n)$ also induces the evident inclusion in homology as a map of $\ca^{\vee}$-comodules.
\end{lem}
\begin{proof}
By the Thom isomorphism, there are isomorphisms
\[H_*(y(n))  \cong H_*(\Omega J_{2^n-1}S^2 ) \text{ and } H_*(z(n)) \cong H_*(\Omega( J_{2^n-1}S^2\langle 2  \rangle) ).\]
The Serre spectral sequence arising from the path-loops fibration has signature
\begin{align}
\label{serre yn}
\mathrm{E}^2_{p,q}= H_p(J_{2^n-1}S^2; H_q(\Omega J_{2^n-1} S^2) ) \Longrightarrow H_{p+q}(P J_{2^n-1} S^2) = \f_2\{1\}
\end{align}
and there is a map of Serre spectral sequences from this spectral sequence to 
\begin{align}
\label{serre hfp} \mathrm{E}^2_{p,q}= H_p(\Omega S^3; H_q(\Omega^2S^3) ) \Longrightarrow H_{p+q}(P \Omega S^3) = \f_2\{1\}.
\end{align}
The latter computation, along with its $\ca^{\vee}$-coaction, follows by~\cite[Thm.~4]{CMT81} and the map of spectral sequences is a map of $\ca^{\vee}$-comodules on the $\mathrm{E}_2$-page
\[ 
P_{2^n-1}(x)\otimes H_*(\Omega J_{2^n-1}S^2)\hookrightarrow  P(x)\otimes H_*(\Omega^2S^3)
\]
which must be a monomorphism by direct inspection of the differentials forced by the triviality of the abutment. The result then follows by naturality of the Thom isomorphism applied to the map~$\mathrm{Th}(\Omega i_{2^n-1})_*\colon \thinspace H_*(y(n))\to H_*(H\f_2)$. The remaining statements are proven in exactly the same way, so we omit the proof. 
\end{proof}

Recall that a finite $2$-local spectrum $F$ has type $m+1$ if $K(m+1)_{*}F\ne 0$ and $K(m)_{*}F=0$ (which implies $K(s)_{*}F=0$ for $0\le s\le m-1$ since $F$ is finite). Following~\cite{MR99}, we say that a $2$-complete spectrum $X$ has \emph{fp type $m$} if there exists a finite $2$-local spectrum $F$ of type $m+1$ such that $\pi_{k}F\otimes X$ is a finite abelian group for each integer $k$ and only nontrivial for finitely many integers $k$.

\begin{cor}\label{not fp type m}
The spectra $y(n)_{2}^{\wedge}$ and~$z(n)_{2}^{\wedge}$ are not fp type~$m$ for any finite~$m$.
\end{cor}

\begin{proof}
The cokernels of the inclusions $H_*(y(n)_{2}^{\wedge})\to \ca^{\vee}$ and~$H_*(z(n)_{2}^{\wedge})\to \ca^{\vee}$ are not finitely generated as $\ca^{\vee}$-comodules, so $H_{*}(y(n)_{2}^{\wedge})$ and~$H_{*}(z(n)_{2}^{\wedge})$ are not finitely presented as comodules over the dual Steenrod algebra. The result then follows by \cite[Prop.~3.2]{MR99}.
\end{proof}

From the homology calculation above, we can also determine that $y(n)$ and $z(n)$ are not highly structured ring spectra in the following sense. 

\begin{cor}
The $\mathbb{E}_1$~algebra structure on $y(n)$ and $z(n)$ cannot be extended to an $\mathbb{E}_2$~algebra structure. 
\end{cor}

\begin{proof}
We give the proof for $y(n)$; the proof for $z(n)$ is similar. An extension of the $\mathbb{E}_1$~algebra structure to an $\mathbb{E}_2$~algebra structure implies an extension of the $H_1$ algebra structure to an $H_2$ algebra structure. If the $\mathbb{E}_{1}$ algebra structure on $y(n)$ can be extended to an $\mathbb{E}_{2}$ algebra structure on $y(n)$, then the operation $Q^{|x|+1}(x)$ on $H_{*}(y(n))$ is well-defined by~\cite[III.3.1,~III.3.2,~III.3.3]{BMMS06} (cf.~\cite[Thm.~5.2]{Law20}). We will show that this leads to a contradiction.

Suppose the $\mathbb{E}_1$~algebra structure of $y(n)$ extends to an $\mathbb{E}_2$~algebra structure. Then we may form a Postnikov truncation in the category of $\mathbb{E}_2$ algebras to produce a map~$y(n)\to H\pi_0(y(n))=H\f_2$. 
This map induces an inclusion~$H_*(y(n))\hookrightarrow \ca$ and therefore sends $\bxi_n\in H_{2^n-1}(y(n))$ to $\bxi_n\in \ca^\vee$. This inclusion must be compatible with $Q^{|x|+1}$. However, we know that $Q^{|\bxi_n|+1}(\bxi_n)=\bxi_{n+1}$ in $H_*(H\mathbb{F}_2)$ by \cite[Thm. 2.2]{BMMS06}, but $\bxi_{n+1}$ is not in the image of the inclusion. We can conclude that the $\mathbb{E}_1$~algebra structure of $y(n)$ cannot be extended to an $\mathbb{E}_2$ algebra structure. For $z(n)$, the argument is the same except that we consider the Postnikov truncation in $\mathbb{E}_2$~algebras $z(n)\to H\pi_0(z(n))=H\mathbb{Z}$, which also induces an inclusion in homology~$H_*(z(n))\hookrightarrow (\ca//E(Q_0))_*$ and the rest of the argument is the same.
\end{proof}

The fact that $y(n)$ is not an $\mathbb{E}_2$~ring spectrum will play a key role in Section~\ref{Sec:THH}. 
\begin{conv}
We will often use the unit map~$S \to y(n)$ as well as the map
\[\mathrm{Th}(\Omega i_{2^n-1})\colon \thinspace y(n) \to H\f_2\] 
induced by the inclusion~$J_{2^n-1} S^2 \hookrightarrow J_\infty S^2 \simeq \Omega S^3$. From this point on, any map~$y(n)\to H\f_2$ without decoration refers to the latter. From this point on we also write $y(n)$ for the $2$-completion of $y(n)$ and $z(n)$ for the $2$-completion of $z(n)$. 
\end{conv}

\begin{lem}\label{Lem:ynF2}
The map~$y(n) \to H\f_2$ is $(2^{n+1}-2)$-connected.
\end{lem}

\begin{proof}
The Adams spectral sequence converging to $y(n)_*$ has the form
\[ 
\Ext^{s,t}_{\ca^{\vee}}(\f_2,P(\bxi_1,\bxi_2,\ldots,\bxi_n)) \Longrightarrow y(n)_{t-s}
\]
and the Adams spectral sequence converging to $(H\f_2)_*$ has the form
\[ 
\Ext^{*,*}_{\ca^{\vee}}(\f_2,P(\bxi_1,\bxi_2,\ldots)) \cong \f_2 \Longrightarrow (H\f_2)_{*} = \f_2.
 \]
Let $C_\bullet(n)$ denote the $\mathrm{E}_1$-page of the Adams spectral sequence for $y(n)$ and let $C_\bullet(\infty)$ denote the $\mathrm{E}_1$-page for the Adams spectral sequence for $H\f_2$. These two $\mathrm{E}_1$-pages differ only in stems above the degree of $\bxi_{n+1}$. Since $|\bxi_{n+1}| = 2^{n+1}-1$, the resulting $\mathrm{E}_2$-pages agree up to stem~$2^{n+1}-2$. Since the second spectral sequence collapses at the $\mathrm{E}_2$-page, we conclude that the map~$\pi_i(y(n)_{2}^{\wedge}) \to \pi_i(H\f_2)$ is an isomorphism for $i <2^{n+1}-2$ and surjective when $i=2^{n+1}-2$.
\end{proof}

\subsection{The localized Adams spectral sequence}\label{Sec:LocASS}
In this section, we recall the localized Adams spectral sequence~\cite{MRS01, Mil81} and discuss its applications to the study of chromatic complexity. Recall that the $m$-th connective Morava K-theory~$k(m)$ has homology~$H_*(k(m)) \cong \ca^{\vee} \square_{E(\bxi_{m+1})} \f_2$ where $\bxi_{m+1}$ is dual to the $m$-th Milnor primitive~$Q_m$. The Adams spectral sequence converging to $k(m)_*(X)$ has the form
\[
\mathrm{E}_2^{s,t} = \Ext^{s,t}_{E(\bxi_{m+1})}(\f_2,H_*(X)) \Longrightarrow k(m)_{t-s}(X)
\]
by the K{\"u}nneth isomorphism and the change-of-rings isomorphism. The spectral sequence collapses for degree reasons when $X$ is the sphere spectrum to show that $k(m)_* \cong P(v_m)$ with $|v_m| = 2^{m+1}-2$.

Morava K-theory~$K(m)$ with $K(m)_* \cong \f_2[v_m^{\pm 1}]$ can then be constructed as the telescope
\[ 
K(m) = \widehat{k(m)} = \hocolim \left(k(m) \overset{v_m}{\longrightarrow} \Sigma^{-(2^{m+1}-2)} k(m) \overset{v_m}{\longrightarrow} \cdots \right) \,.
\]
Since smash product commutes with filtered colimits, the spectrum~$K(m) \wedge X$ is the telescope of the self-map~$v_m \wedge id_X$ on $k(m) \wedge X$. 

The homotopy groups of a telescope can sometimes be computed using the localized Adams spectral sequence introduced in~\cite{Mil81}. Our recollection follows~\cite{MRS01}. 

\begin{const}
Let $Y$ be a spectrum with a $v_n$ self-map $f \co Y \to \Sigma^{-d}Y$, let $\widehat{Y}$ be the telescope of $f$, and let
\begin{equation}\label{eq:Adams-resolution-Y}
Y = Y_0 \leftarrow Y_1 \leftarrow Y_2 \leftarrow \cdots
\end{equation}
be an Adams resolution of $Y$. Suppose there is a lifting~$\tilde{f} : Y \to \Sigma^{-d} Y_{s_0}$ for some~$s_0 \geq 0$. This lifting induces maps~$\tilde{f} : Y_s \to \Sigma^{-d} Y_{s + s_0}$ for each $Y_s$ in the Adams resolution \eqref{eq:Adams-resolution-Y}. Iterate these maps to define telescopes~$\widehat{Y_s}$ for $s\geq 0$ and set $Y_s = Y$ for $s < 0$ to produce a tower
\[
\cdots \leftarrow \widehat{Y_{-1}} \leftarrow \widehat{Y_0} \leftarrow \widehat{Y_1} \leftarrow \cdots \,.
\]
The resulting conditionally convergent full plane spectral sequence
\[ 
v_m^{-1} \Ext^{s,t}_{\ca^{\vee}}(\f_2,H_*(Y)) \Longrightarrow \pi_{t-s}(\widehat{Y})
\]
is the \emph{localized Adams spectral sequence.}
\end{const}

\begin{thm}[{\cite[Thm.~2.13]{MRS01}}]\label{MRS thm}
For a spectrum~$Y$ equipped with maps~$f$ and $\tilde{f}$ as above, in the localized Adams spectral sequence for $\pi_*(\widehat{Y})$ we have
\begin{itemize}
\item The homotopy colimit~$\underset{s}{\hocolim \thinspace} \widehat{Y}_{-s}$ is the telescope~$\widehat{Y}$.
\item The homotopy limit~$\underset{s} {\holim \thinspace}\widehat{Y_s}$ is contractible if the original (unlocalized) Adams spectral sequence has a vanishing line of slope $\frac{s_0}{d}$ at $E_r$ for some finite $r$, i.e. if there are constants $c$ and~$r$ such that 
\[
E^{s,t}_r  =0 \quad \text{ for } \quad s > c+(t-s)(s_0/d) \,.
\]
(In this case, we say $f$ has a \emph{parallel lifting}~$\tilde{f}$.)
\item If $f$ has a parallel lifting, this localized Adams spectral sequence strongly converges to $\pi_*(\widehat{Y})$. 
\end{itemize}
\end{thm}

\begin{proof}
Note that although Mahowald--Ravenel--Shick work at odd primes, the proof of the result~\cite[Thm.~2.13]{MRS01} does not depend on the prime and therefore carries over \emph{mutatis mutandis}. We spell this out as follows. First, 
\[ \colim_{s} \widehat{Y}_{s}\simeq \colim_{s} \colim_{i} \Sigma^{-di} Y_{s+is_0}\simeq   \colim_{i} \colim_{s}\Sigma^{-di} Y_{s+is_0}\simeq \colim_i \Sigma^{-di}Y \simeq  \widehat{Y} \,.\]
Letting $\mathrm{E}_{r}^{s,t}(Y)$ be the $\mathrm{E}_{r}$-page of the Adams spectral sequence for $Y$ and $\mathrm{E}_{r}^{s,t}(\widehat{Y})$ be the localized Adam spectral sequence, we note that 
\[ \mathrm{E}_{r}^{s,t}(\widehat{Y})=\lim_{k}\mathrm{E}_{r}^{s+ks_{0},t+kd}(Y) \]
so the vanishing line for the localized Adams spectral sequence follows from the vanishing line for the unlocalized Adams spectral sequence. 

Let $g: Y_{s}\to Y_{s-r+1}$ be the structure map in the Adams resolution \eqref{eq:Adams-resolution-Y} for $Y$. The vanishing line implies $\pi_{m}(g)=0$ for $m<(sd+c)/s_{0}$.  Compatibility of $g$ with the map $Y_{s}\to \Sigma^{-dk}Y_{s+s_{0}k}$ implies that $g :\widehat{Y}_{s}\to \widehat{Y}_{s+r-1}$ has the same property. Therefore, fixing $m$ and $s$, the map 
\begin{equation}\label{eq:key-map}
\pi_{m}\widehat{Y}_{i}\to \pi_{m}\widehat{Y}_{s}
\end{equation} 
is trivial for large enough $i$ and the image of $\lim_{i}\pi_{*}\widehat{Y}_{i}\to \pi_{*}\widehat{Y}_{s}$ is trivial for each $s$, so $\lim_{i}\pi_{*}\widehat{Y}_{i}=0$. The description of the map \eqref{eq:key-map} above implies that for each $m$ the sequence $\{\pi_{m}\widehat{Y}_{i}\}$ satisfies the Mittag--Leffler condition and consequently $R^{1}\lim_{i}\pi_{*}\widehat{Y}_{i}$ is zero. We can then apply~\cite[Lemma~8.1]{Boa99}, to determine that the spectral sequence strongly converges by~\cite[Theorem~8.2]{Boa99}. 
\end{proof}

\begin{rem}
Mahowald--Ravenel--Shick compute $v_n^{-1}\mathrm{E}_2$ for the localized Adams spectral sequence converging to $\pi_*(\widehat{y(n)})$ in~\cite[Sec.~2.3]{MRS01}. Their computations show that the localized Adams spectral sequence for $\pi_*(\widehat{y(n)})$ strongly converges. 
\end{rem}

As suggested by the notation, the $\mathrm{E}_2$-page of the localized Adams spectral sequence can be computed by inverting $v_m$ at the level of $\Ext$-groups as in \cite[Sec.~2.5]{MRS01} and~\cite{Eis88}. 

We now specialize to the case $Y = k(m) \wedge X$ and $f = v_m \wedge id_X$ so that $\widehat{Y} \simeq K(m) \wedge X$. Note that $f$ lifts to $\tilde{f} : Y \to \Sigma^{-2^{m+1}+2} Y_1$ since $v_m$ has Adams filtration one; indeed, we may take $s_0=1$ with $\tilde{f} := \tilde{v}_m \wedge \id_X$ where $\tilde{v}_m$ is a lift of $v_m$. By applying the K{\"u}nneth isomorphism and a change-of-rings isomorphism, we see that the localized Adams spectral sequence takes the form
\begin{align}
\label{localizedadams}
v^{-1}_m \mathrm{E}_2 := v^{-1}_m\Ext^{s,t}_{E(\bxi_{m+1})}(\f_2,H_*(X)) \Longrightarrow  K(m)_{t-s}(X).
\end{align}
Note that we have only used the $v_m$ self-map on $k(m)$, so there is no decoration on $X$. 

In order to compute $v^{-1}_m \mathrm{E}_2$, we will use Margolis homology~\cite[Ch. 19]{Mar11} which encodes the action of the Milnor primitive~$Q_m$ on $H_*(X)$.

\begin{defin}\label{def Q_m}
Let $M$ be a module over $E(Q_m)$. Since $Q_m^2=0$, we obtain a complex
\[
\dots \overset{Q_m}{\longrightarrow} M \overset{Q_m}{\longrightarrow} M \overset{Q_m}{\longrightarrow} M \overset{Q_m}{\longrightarrow} \dots \,.
\]
The homology
\[ 
H(M;Q_m):= \ker (Q_m) /\im (Q_m)
 \]
is the \emph{Margolis homology of $M$ with respect to $Q_m$}. 
\end{defin}

\begin{lem}\label{me2}
Suppose $H_*(X)$ is bounded below and finite type. There is an isomorphism
\begin{align*}
 \Ext_{E(\bar{\xi}_{m+1})}^{*,*}(\f_2,H_*(X)) \cong H(H_*(X);Q_m)\otimes P(v_m) \oplus T
\end{align*}
where $|v_m| = (1,2^{m+1}-1)$ and $T$ is a simple $v_m$-torsion module concentrated in bidegrees~$(0,t)$ with $t \geq \min\{i : H_i(X) \neq 0\}$. 
\end{lem}
\begin{proof}
Since $E(\bar{\xi}_{m+1})$ is a finite dimensional connected Hopf algebra, the category of $E(\bar{\xi}_{m+1})$-comodules has enough projectives~\cite[Thm.~10]{Lin77}. To compute 
\[
\Ext_{E(\bar{\xi}_{m+1})}^{*,*}(\f_2,H_*(X)) 
\]
we therefore consider the minimal projective resolution of $\f_2$ in the category of $E(\bar{\xi}_{m+1})$-comodules 
\[ 
\f_2 \leftarrow E(\bar{\xi}_{m+1}) \leftarrow E(\bar{\xi}_{m+1}) \leftarrow 
\]
where the first map is the canonical quotient and the remaining maps are all given by multiplication by $\bar{\xi}_{m+1}$. Truncating and applying $\Hom_{E(\bar{\xi}_{m+1})}(-,H_*(X))$ produces the sequence 
\begin{align}\label{margolis resolution} 
0 \longrightarrow H_*(X)\longrightarrow H_*(X)\longrightarrow \dots 
\end{align}
where all the maps are induced by the $E(\overline{\xi}_{m+1})$-coaction on $H_*(X)$. Note that the dual $E(\overline{\xi}_{m+1})^{\vee}:=\mathrm{Hom}_{\mathbb{F}_{p}}(E(\overline{\xi}_{m+1}),\mathbb{F}_{p})$ can be identified with $E(Q_m)$ and we can regard a left $E(Q_m)$-comodule $M$ with coaction $\nu(m)=\sum_{i=0}^k m_{i,0}\otimes m_{i,1}$ as a (right) $E(Q_{m})$-module with (right) action $\phi(m\otimes f)=\sum_{i=0}^k f(m_{i,0})\cdot m_{i,1}$, see~\cite[Ch.~1~4.1(1)]{BW03}. In fact, the category of $E(\overline{\xi}_{m+1})$-comodules is equivalent to the category of $E(Q_m)$-modules by~\cite[Ch.~1~4.7(e)]{BW03} , so we can view this sequence as a sequence of $E(Q_m)$-modules. Therefore, we can identify 
\[
\Ext_{E(\bar{\xi}_{m+1})}^{s,*}(\f_2,H_*(X))\cong H(H_*(X);Q_m)
\]
when $s>0$ and~$\Ext_{E(\bar{\xi}_{m+1})}^{0,*}(\f_2,H_*(X))\cong \ker (Q_n\colon \thinspace H_*(X)\longrightarrow H_*(X))$. Recalling that 
\[
\Ext_{E(\overline{\xi}_{m+1})}^{*,*}(\f_2,\f_2)\cong P(v_m) 
\] 
with $|v_m| = (2^{n+1}-1,1)$, we see that the $P(v_m)$-module structure on $\Ext_{E(\bar{\xi}_{m+1})}^{s,*}(\f_2,H_*(X))$ is induced by the Yoneda pairing in $\Ext$, i.e., by pairing the chain complex \eqref{margolis resolution} for $H_*(X)$ as in the lemma statement with the same chain complex where $X$ is the sphere spectrum. This pairing evidently induces the isomorphism of $P(v_m)$-modules described in the statement of the lemma where $T$ consists of elements in the kernel of $Q_n\colon \thinspace H_*(X)\longrightarrow H_*(X)$. 
\end{proof}

\begin{cor}\label{mvassv}
Let $m \geq 1$ and suppose $X$ is a bounded below spectrum and $H_*(X)$ is finite type. The following statements hold: 
\begin{enumerate}
\item \label{item 1 mvassv} If $H(H_*(X);Q_m) = 0$, then $v^{-1}_m\mathrm{E}_2 = 0$. 
\item \label{item 2 mvassv} The $\mathrm{E}_2$ page of \eqref{localizedadams}, denoted $v_m^{-1}\mathrm{E}_2$, has a vanishing line of slope~$1/|v_m|$. 
\item \label{item 3 mvassv} The localized Adams spectral sequence associated to $k(m) \wedge X$ with the self-map~$v_m \wedge id_X$ strongly converges to $K(m)_*(X)$.
\end{enumerate}
\end{cor}

\begin{proof}
Parts \eqref{item 1 mvassv} and~\eqref{item 2 mvassv} are clear from Lemma~\ref{me2}. Statement~\eqref{item 3 mvassv} follows by applying Theorem~\ref{MRS thm}. 
\end{proof}
\begin{rem}
When all the hypotheses for Corollary~\ref{mvassv} hold including $H(H_*(X);Q_m) = 0$ then it is a consequence that $K(m)_*X\cong 0$. 
\end{rem}

\subsection{Chromatic complexity of $y(n)$, $z(n)$, and $z(n)/v_n$}\label{Sec:ynznCC}
We now apply the localized Adams spectral sequence to determine the chromatic complexity of $y(n)$, $z(n)$, and a spectrum we define in this section~$z(n)/v_n$. 

The action of $Q_m$ on the generator~$\bxi_k \in \ca^{\vee}$ can be computed using the coproduct~$\psi\co  \ca^{\vee} \to \ca^{\vee} \otimes \ca^{\vee}$ defined in~\eqref{coprod dual steenrod}. In particular, we have 
\begin{equation}\label{Qaction} 
Q_m(\bxi_k) = \begin{cases} 
				\bxi_{k-m-1}^{2^{m+1}} \quad & k \geq m+1,\\
				0 \quad &else,
			\end{cases}
\end{equation}
where $\bxi_0=1$. This action can be extended to all of $\ca^{\vee}$ using the fact that $Q_m$ acts as a derivation. 

As a warm-up for later computations, we compute the Margolis homology of the dual Steenrod algebra. The chain complexes defined in the proof will be used in our computation of $H(H_*(y(n));Q_m)$ below. 

\begin{lem}\label{mhfp}
The Margolis homology of the dual Steenrod algebra~$H(\ca^{\vee};Q_m)$, or equivalently the Margolis homology of $H\f_2$, vanishes for all $m \geq 0$. 
\end{lem}

\begin{proof}
This is~\cite[Ch.~19, Prop.~1]{Mar11}, but our proof is modeled after~\cite[Lem.~16.9]{Ada95}. We begin with $H(\ca^{\vee};Q_0)$, which is somewhat exceptional. Express $\ca^{\vee}$ as the tensor product of the chain complexes (with differential~$Q_0$)
\begin{align*}
(e_0) \quad &\f_2\{1\} \leftarrow \f_2\{\bxi_1\}, \\
(c_r) \quad &\f_2\{1, \bxi_r^2\} \leftarrow \f_2\{\bxi_{r+1},\bxi^4_r\} \leftarrow \f_2\{\bxi_r^2\bxi_{r+1}, \bxi^6_r\} \leftarrow \f_2\{\bxi^4_r\bxi_{r+1},\bxi^8_r\} \leftarrow \cdots,
\end{align*}
where $r \geq 1$.~Each chain complex~$(c_r)$ has homology~$\f_2\{1\}$ and the chain complex~$(e_0)$ has vanishing homology, so by the K{\"u}nneth isomorphism for Margolis homology~\cite[Ch.~19, Prop.~18]{Mar11}, we have $H(\ca^{\vee};Q_0) \cong 0$. 

Now we compute $H(\ca^{\vee};Q_1)$. Decompose $\ca^{\vee}$ as the tensor product of the chain complexes (with differential~$Q_1$)
\begin{align*}
(e_0) &\quad \f_2\{1\} \leftarrow \f_2\{\bxi_2\} \,, \\
(c_r) &\quad \f_2\{1, \bxi_r^4\} \leftarrow \f_2\{\bxi_{r+2},\bxi^8_r\} \leftarrow \f_2\{\bxi^4_r\bxi_{r+2}, \bxi^{12}_r\} \leftarrow \f_2\{\bxi^8_r \bxi_{r+2},\bxi^{16}_r\} \leftarrow \cdots \,, \\
(d_1) &\quad \f_2\{1,\bxi_1,\bxi^2_1, \bxi_1^3\} \,,\\
(d_s) &\quad \f_2\{1,\bxi_s^2\} \,,
\end{align*}
where $r \geq 1$ and $s \geq 2$. The chain complex~$(e_0)$ has vanishing homology, so by the K{\"u}nneth isomorphism we have $H(\ca^{\vee};Q_1) \cong 0$.

The computation of $H(\ca^{\vee};Q_m)$ for $m \geq 2$ is similar. Decompose $\ca^{\vee}$ into chain complexes as above; the chain complex
\[
(e_0) \quad \f_2\{1\} \leftarrow \f_2\{\bxi_{m+1}\}
\]
has vanishing homology, so $H(\ca^{\vee};Q_m) = 0$. 
\end{proof}

\begin{cor}\label{mhz}
The Margolis homology of $(\ce //E(0))^\vee \cong \f_2[\bxi^2_1,\bxi_2,\ldots]$, or equivalently the Margolis homology of $H\z$, is given by
\[
H(H_*(H\z);Q_m) \cong \begin{cases} \f_2 \quad &m=0 ,\\
0 & else \,.
\end{cases}
\]
\end{cor}
\begin{proof}
We begin with $m=0$. The only difference between this computation and the computation for $H_*(H\f_2,Q_0)$, is that we omit the chain complex $(e_0)$. Since the homology of the remaining complexes~$(c_r)$ is $\f_2$ in each case, $H_*(H\z,Q_0)\cong \f_2$. 

For $m > 0$, we can use the same complexes as in the previous proof after replacing $(d_1)$ by the chain complex $\f_2\{1,\bxi^2_1\}$.
\end{proof}

We  compute the Margolis homology of $y(n)$ and~$z(n)$ by modifying these complexes further.

\begin{lem}\label{mhyn}
The Margolis homology of $P(\bxi_1,\ldots,\bxi_n)$, or equivalently the Margolis homology of $y(n)$, is given by
\[
H(H_*(y(n));Q_m) \cong \begin{cases}
0 \quad &\text{ if } 0 \leq m \leq n-1 \,, \\
H_*(y(n)) \quad &\text{ if } m \geq n \,.
\end{cases}
\]
\end{lem}

\begin{proof}
When $m=0$, the first $r$ for which the complex~$(c_r)$ cannot be defined is $(c_n)$ since $\bxi_{n+1} \notin H_*(y(n))$. Therefore we replace $(c_n)$ by the complex
\[
(c_n') \quad \f_2\{1\} \leftarrow \f_2\{\bxi^2_n\} \leftarrow \f_2\{\bxi^4_n\} \leftarrow \f_2\{\bxi^6_n\} \leftarrow \cdots \,.
\]
Then $H_*(y(n))$ decomposes as the tensor product of the complexes$(e_0)$, $(c_r)_{1 \leq r \leq n-1}$, and~$(c_n')$. The homology of $(c_n')$ is nontrivial but the homology of $(e_0)$ vanishes, so $H(H_*(y(n));Q_0) \cong 0$.

When $1 \leq m \leq n-1$, we make a similar change. We end up with redefined chain complexes~$(c_r')$ for $n-m \leq r \leq n$. Since we still tensor with the acyclic complex~$(e_0)$, we still have $H(H_*(y(n));Q_m) = 0$. 

When $m \geq n$, we no longer include the chain complex~$(e_0)$ since $\bxi_{n+1} \notin H_*(y(n))$. Since $Q_n(\bxi_i) = 0$ for all $1 \leq i \leq n$, we see that $H_*(y(n))$ is generated by cycles and obtain the desired isomorphism.
\end{proof}

The same techniques adapted to the complexes used to compute $H(H_*(H\z);Q_m)$ give the following:

\begin{cor}\label{mhzn}
The Margolis homology of $P(\bxi^2_1,\bxi_2,\ldots,\bxi_n)$, or equivalently the Margolis homology of $z(n)$, is given by
\[
H(H_*(z(n));Q_m) \cong \begin{cases}
P(\bxi_n^2) \quad&\text{ if } m=0 \,, \\
0 \quad &\text{ if } 1 \leq m \leq n-1 \,, \\
H_*(z(n)) \quad &\text{ if } m \geq n \,.
\end{cases}
\]
\end{cor}
\begin{proof}
We do the same alterations to Lemma \ref{mhyn} as we did to produce the proof of Corollary \ref{mhz} from Lemma \ref{mhfp}, so we will just describe the case $m=0$. We use the same chain complexes as in Lemma \ref{mhyn} except we omit the acyclic complex $(e_0)$. Thus, the Margolis homology is a tensor product of copies of $\f_2$ with the homology of $(c_n^{\prime})$. Thus, $H(H_*(z(n));Q_0)\cong P(\bxi_n^2)$.
\end{proof}

The following lemma will be useful for further describing the chromatic complexity of $z(n)$.
\begin{lem}\label{znsm}
The spectrum~$z(n)$ has a self map 
\[ 
v_n\colon \Sigma^{2p^n-2}z(n) \to z(n)
\]
that induces the zero map on $K(m)_*$ for $1\le m<n$ and the zero map on mod~$2$ homology $H_*$. Moreover,  we have an isomorphism of $\ca^{\vee}$-comodules
\[
H_*(z(n)/v_n) \cong H_*(z(n)) \otimes E(\bxi_{n+1}) \,.
\]
\end{lem}

\begin{proof}
We analyze the Adams spectral sequence for $z(n)$ in a range. First we describe the input of the Adams spectral sequence. Following~\cite{MRS01}, write $B(n)_*$ for the quotient Hopf algebra 
$B(n)_*=\ca^{\vee}/(\bxi_1,\dots, \bxi_n)$ such that 
\[ 
H_*(y(n))=\ca^{\vee}\square_{B(n)_*}\f_2 \,.
\]
Let $C(n)_*$ denote the quotient Hopf algebra~$C(n)_*=\ca^{\vee}/(\bxi_1^2,\dots \bxi_n)$ such that 
\[
H_*(z(n))= \ca^{\vee}\square_{C(n)_*}\f_2 \,.
\]
There is a Hopf algebra extension 
\[
B(n)_* \longrightarrow  C(n)_*\longrightarrow E(\bxi_1) 
\] 
and an associated Cartan--Eilenberg spectral sequence 
\[ 
\Ext_{B(n)_*}^{s}(\mathbb{F}_2,\Ext_{E(\bxi_1)_*}^t(\mathbb{F}_2,\mathbb{F}_2) ) \cong \Ext^{s}_{B(n)_*}(\mathbb{F}_2,P(h_0)_t)\Longrightarrow \Ext^{s+t}_{C(n)_*}(\mathbb{F}_2,\mathbb{F}_2) 
\]
where we write $P(h_0)_t$ for the degree $t$ part of the graded ring $P(h_{0})$. The Cartan--Eilenberg spectral sequence collapses to the $(s,0)$-line because $|h_0|=|\xi_1|-1=0$. Since there are no classes in $\Ext_{B(n)_{*}}^k(\mathbb{F}_2,\mathbb{F}_2)$ in adjacent degrees for $k\le 2^{n+1}$ and the Adams spectral sequence for $y(n)$ collapses in this range by~\cite[Lem.~3.5]{MRS01}, the Adams spectral sequence for $z(n)$ also collapses in this range. Again, by~\cite[Lem. 3.5]{MRS01} there is an element~$v_n$ in Adams filtration one in $\Ext^{*,*}_{C(n)_*}(\mathbb{F}_2,\mathbb{F}_2)$ and we observe that it supports an $h_0$-tower. Consequently, there is an element~$v_n\in \pi_{2^{n+1}-2}(z(n))$ generating the $2$-adic integers~$\mathbb{Z}_2^{\wedge}$. 

Since $z(n)$ is an $\mathbb{E}_1$ ring spectrum we can produce a self map as the composite 
\[ 
S^{2^{n+1}-2}\wedge z(n) \xrightarrow{v_n \wedge \id_{z(n)}} z(n) \wedge z(n)_{2}^{\wedge} \to z(n) \,.
\]
This is also the map obtained by taking the adjoint to $v_n\co  S^{2^{n+1}-1}\to z(n)_{2}^{\wedge}$ in the category of (right) $z(n)$-modules. By Corollary~\ref{mhzn}, this map induces the zero map on $K(m)_*$ for $1\le m <n$, since it implies that the source and target of the map 
\[ 
\begin{tikzcd}
0 = K(m)_{*}\Sigma^{2p^{n}-2}z(n) \arrow{rr}{K(m)_{*}v_{n}} && K(m)_{*}z(n) =0
\end{tikzcd}
 \]
are zero using Corollary~\ref{mvassv} and therefore it can only be the zero map. This will be explained in more detail in Proposition~\ref{cczn}. It also induces the zero map on $H_*$ because $v_n$ is detected by an element in Adams filtration one.  

Therefore we have an extension  
\[ 
0 \to H_*(z(n)) \to H_*(z(n)/v_n) \to  \Sigma^{2^{n+1}-1}H_*(z(n)) \to 0 
\]
of $\ca^{\vee}$ comodules. The group of possible $\ca^{\vee}$-comodule extensions is given by
\[ 
\Ext_{\ca^{\vee}}^1(\Sigma^{2^{n+1}-1}H_*(z(n)),H_*(z(n)))\cong  \Ext_{C(n)_*}^1(\Sigma^{2^{n+1}-1}H_*(z(n)),\mathbb{F}_2) \,,
\]
using a change of rings isomorphism and the isomorphism~$H_*(z(n))\cong \ca^{\vee}\square_{C(n)_*}\mathbb{F}_2$. 

We are therefore reduced to examining the possible $C(n)_*$-comodule extensions
\[ 
0\to \mathbb{F}_2 \to E \to \Sigma^{2^{n+1}-1}H_*(z(n))\to 0 \,,
\]
but all such extensions of $C(n)_*$-comodules are trivial by examination of the grading preserving coaction. This implies $H_*(z(n)/v_n)\cong H_*(z(n))\otimes E(x)$ where $|x|=2^{n+1}-1$. 

We now use the Adams spectral sequence to determine the $\ca^{\vee}$-coaction on $x$. Note that $z(n)/v_n$ was obtained by coning off the element in homotopy detected by the permanent cycle~$\bxi_n\otimes 1$ in the Adams spectral sequence with
\[
\mathrm{E}_2 = \Ext_{C(n)_*}^{*,*}(\mathbb{F}_2, E(x)) \,.
\]
The only element that can kill $\bxi_n\otimes 1$ is $x$, so we have
\[
d_1(x)=\bxi_{n+1}\otimes 1 \,.
\]
On the other hand, the $d_1$-differentials in the Adams spectral sequence can be calculated using the formula for differentials in the cobar complex, so $d_1(x)=1\otimes x -\psi_n(x)$ where $\psi_n(x)$ is the coaction of $x$ in $H_*(z(n)/v_n)$. We therefore see that 
\[ 
\psi_n(x)=\bxi_{n+1}\otimes 1+1\otimes x+\text{($d_1$-boundaries)} \,.
\]

Finally, since the composite~$\Sigma^{2^{n+1}-2}z(n)\overset{v_n}{\to}z(n)\to z(\infty)=H\mathbb{Z}$ is nullhomotopic, there is a map~$z(n)/v_n\to H\mathbb{Z}$. This map sends $\bxi_{n+1}\otimes 1$ to the class with the same name in the cobar complex for $H\mathbb{Z}$. In the latter cobar complex, $\bxi_{n+1}\otimes 1$ is killed by a differential on $\bxi_{n+1}\in H_*(H\mathbb{Z})$. Therefore $x$ maps to $\bxi_{n+1}$ under the map of spectral sequences. Since the map~$H_*(z(n)/v_n)\to H_*(H\mathbb{Z})$ is a map of $\ca^{\vee}$-comodules, the coaction on $x$ coincides with the coaction on $\bxi_{n+1}$. Note also that there is no room for hidden comodule extensions because $|x|>|y|$ for all generators~$y$ of $H_*(z(n)/v_n)$. 
\end{proof} 

\begin{rem}
In fact, the spectrum~$z(n)/v_n$ may be constructed as the Thom spectrum of the map 
\[ S^{2p^n-1}\to B \GL_1(z(n))\]
adjoint to $v_n\in\pi_{2p^n-2}(\GL_1(z(n)))\cong \pi_{2p^n-2}(z(n))$. The first author would like to thank Jeremy Hahn for pointing this out. 
\end{rem}

We now determine the chromatic complexity of $z(n)/v_n$. 
\begin{cor}\label{znmvn}
The Margolis homology of $P(\bxi^2_1,\bxi_2,\ldots,\bxi_n) \otimes E(\bxi_{n+1})$, or equivalently the Margolis homology of $z(n)/v_n$, is given by
\[
H(H_*(z(n)/v_n);Q_m) \cong \begin{cases}
\f_2 \quad&\text{ if } m=0 \,, \\
0 \quad &\text{ if } 1 \leq m \leq n \,, \\
H_*(z(n)/v_n) \quad &\text{ if } m \geq n+1 \,.
\end{cases} 
\]
\end{cor}
\begin{proof}
The proof in the case~$m=n$ follows by tensoring the complexes from the previous corollary with the complex
\[
\f_2\{1\} \leftarrow \f_2\{\bxi_{n+1}\} \,.
\]
Since this complex is $Q_n$-acyclic, we observe that $H(H_*(z(n)/v_n);Q_n)\cong 0$. For $m=0$, we make the following adjustment. Rather than replacing $(c_n)$ with $(c_n')$ as in Lemma~\ref{mhyn}, we keep the complex~$(c_n)$ and remove $(c_r)$ for $r>n$. This has the consequence that $H_*(z(n)/v_n,Q_0)\cong \f_2$. In the case~$0<m<n$, we only replace $(c_r)$ with $(c_r')$ for $n-m<r\le n$. The Margolis homology is still trivial because we are tensoring with the acyclic complex~$\f_2\{1\} \leftarrow \f_2\{\bxi_m\}$. The case~$m>n$ is exactly the same as in Lemma~\ref{mhyn}
\end{proof}
 
We can assemble these Margolis homology computations to study $K(m)_*(X)$ for $X = y(n)$, $z(n)$, and~$z(n)/v_n$.

\begin{prop}\label{cczn}
The chromatic complexity of $y(n)$, $z(n)$ and~$z(n)/v_n$ may be described as follows: 
\begin{itemize}
\item The spectrum~$y(n)$ is $K(m)$-acyclic for $0 \leq m \leq n-1$, and~$K(n)_*(y(n)) \neq 0$. 
\item The spectrum~$z(n)$ is $K(m)$-acyclic for $1 \leq m \leq n-1$, and~$K(m)_*(z(n)) \neq 0$ for $m=0,n$. 
\item The spectrum~$z(n)/v_n$ is $K(m)$-acyclic for $1 \leq m \leq n$, and~$K(0)_*(z(n)/v_n) \neq 0$.
\end{itemize}
\end{prop} 

\begin{proof}
We begin with showing $K(m)$-acyclicity. First, note that $\pi_{0}y(n)=\mathbb{F}_{p}$ so the homotopy groups of $y(n)$ are torsion and therefore $K(0)_{*}y(n)=0$. 
Let $R \in \{y(n), z(n), z(n)/v_n\}$. Since $R$ is connective, the localized Adams spectral sequence converges to $K(m)_*(R)$ by Corollary~\ref{mvassv} for $m\ge 1$. By the same corollary, we have
$$v_m^{-1}E_2 = v_m^{-1}\Ext_{E(\overline{\xi}_{m+1})}^{*,*}(\f_2,H_*(R)) = 0$$
whenever $H(H_*(R);Q_m)$ vanishes. Lemmas~\ref{mhyn} and Corollaries~\ref{mhzn} and~\ref{znmvn} prove that these Margolis homology groups vanish for the claimed ranges of $m$. 

We now argue that $K(0)_*(z(n))$ and $K(0)_*(z(n)/v_n)$ are nonzero. Recalling that $K(0) = H\mathbb{Q}$ is rational homology, it suffices to produce a torsion-free summand in the homotopy groups of $z(n)$ and $z(n)/v_n$. The $z(n)$-analogue of Lemma~\ref{Lem:ynF2} implies that the map $z(n) \to H\mathbb{Z}_{2}$ induces an isomorphism on $\pi_0$, so $\pi_0(z(n)) \cong \mathbb{Z}_{2}$ and thus $K(0)_*(z(n)) \neq 0$. Similarly, considering the Adams spectral sequences converging to the map $z(n) \to z(n)/v_n$ shows that $\pi_0(z(n)) \cong \pi_0(z(n)/v_n)$, so $K(0)_*(z(n)/v_n) \neq 0$. 

Finally, we argue that for $R \in \{y(n), z(n)\}$ that $K(n)_*(R) \neq 0$. Since $K(n)$ and $R$ are both $\mathbb{E}_1$~ring spectra, the Atiyah--Hirzebruch spectral sequence
\[
\mathrm{E}^{2}_{s,t}=H_s(R;K(n)_t) \Longrightarrow K(n)_{s+t}(R)
\]
is multiplicative, with multiplicative generators all either on the zero line or the zero column (using homological Serre grading). Since we are using the homological Serre grading the differentials are of the form $d_{r}:\mathrm{E}^{r}_{s,t}\to \mathrm{E}^{r}_{s-r,t+r-1}$. The spectral sequence is a right half-plane spectral sequence, so the generators on the zero column cannot support differentials. By Lemma~\ref{lem:homology}, the algebra generators in the zero line are all in degrees less than or equal to $2^n-1$. Since $|v_n| = 2^{n+1}-2$ and $v_{n}$, the $\mathrm{E}^2$-page is isomorphic to the $\mathrm{E}^{2^{n+1}-1}$-page and thus the spectral sequence collapses. Consequently, there is an isomorphism
\[
K(n)_*(R) \cong K(n)_* \otimes H_*(R) \neq 0 .
\]
\end{proof}

\section{Homology of topological Hochschild homology of $y(n)$}\label{Sec:THH}
We now turn to the study of the topological Hochschild homology of $y(n)$. We begin by computing $H_*(\THH(y(n)))$ using the B{\"o}kstedt spectral sequence \cite{BokZ}. We then analyze the map~$\phi_n : H_*(\THH(y(n))) \to H_*(\THH(H\f_2))$. 

\begin{rem}
The calculations in this section and the sequel are complicated by the fact that the spectrum~$\THH(y(n))$ does not admit a ring structure since $y(n)$ is an $\mathbb{E}_1$ ring spectrum, but not an $\mathbb{E}_2$~ring spectrum. Therefore we will only prove \emph{additive isomorphisms} throughout the remaining sections since there is no multiplicative structure on $H_*(\THH(y(n)))$. 
\end{rem}

We will use the map~$(\phi_n)_*: H_*(\THH(y(n))) \to H_*(\THH(H\f_2))$ to name the classes in $H_*(\THH(y(n)))$. This map can be understood modulo intermediary filtration in the sense of Remark~\ref{Rmk:IntermediaryFiltration}. We thank an anonomous referee for sharing this remark, which helps to clarify many of our later arguments. 

\begin{rem}\label{Rmk:IntermediaryFiltration}
Recall that for any $\mathbb{E}_1$~ring spectrum $R$, we have $\THH(R) \simeq |B_\bullet^{cyc}(R)|$, where $B_\bullet^{cyc}$ is the cyclic bar construction. The \emph{B{\"o}kstedt filtration} defines a filtered spectrum~$F_\bullet \THH(R)$, which is obtained by setting $F_n \THH(R) := |\text{sk}_n(B^{cyc}_{\bullet} R)|$ where $\text{sk}_n$ is the $n$-skeleton functor. The homology of this filtered spectrum is then a filtered graded abelian group, with $F_n H_*(\THH(R)) = \im ( H_*(F_n \THH(R))\to H_{*}(\THH(R)))$. 
Let 
\begin{equation}\label{Eqn:BokstedtFiltration}
F_0 \subseteq F_1 \subseteq F_2 \subseteq \cdots \subseteq H_*(\THH(\f_2))
\end{equation}
denote this filtration in the case $R=H\f_2$. The map~$(\phi_n)_*: H_*(\THH(y(n))) \to H_*(\THH(H\f_2))$ is an injective map of filtered graded abelian groups. We will name a class in $H_*(\THH(y(n)))$ which maps nontrivially to the complement~$F_k \setminus F_{k-1}$ by its name in the quotient~$F_k / F_{k-1}$. Therefore, by definition, a class~$x \in H_*(\THH(y(n)))$ maps to the class with the same name in $H_*(\THH(H\f_2))$, modulo lower B{\"o}kstedt filtration.

In fact, we can say slightly more. Since $H\f_2$ is an $\mathbb{E}_{\infty}$~ring spectrum, the $0$-simplices of $B_\bullet^{cyc}(H\f_2)$ split off, and we can rewrite the filtration~\eqref{Eqn:BokstedtFiltration} as
\[
F_0 \subseteq F_0 \oplus \bar{F}_1 \subseteq F_0 \oplus \bar{F}_2 \subseteq \cdots \subseteq H_*(\THH(H\f_2))\,.
\]
Therefore, any class~$x \in H_*(\THH(y(n)))$ mapping to the complement~$F_k \setminus F_{k-1}$ has a well-defined image in $F_0 \oplus F_k/F_{k-1}$. In this sense, we understand classes in $H_*(\THH(y(n)))$ modulo \emph{intermediary filtration}. 

In particular, classes in B{\"o}kstedt filtrations $0$ and~$1$ of $H_*(\THH(y(n)))$ have no indeterminacy modulo intermediary filtration. Further, we obtain a homomorphism
\[
f_0 : H_*(\THH(y(n))) \to F_0 = \ca^{\vee}
\]
which can be used to access the $\ca^{\vee}$-comodule structure of $H_*(\THH(y(n)))$ as we discuss next.
\end{rem}

The $\ca^{\vee}$-coaction on $H_*(\THH(y(n)))$ will be denoted 
\[ 
\nu_n\co H_*(\THH(y(n))) \to \ca^{\vee} \otimes  H_*(\THH(y(n))) 
\]
for $0\le n\le \infty$ with the convention that $y(\infty)=H\mathbb{F}_2$.

\begin{prop}\label{Prop:HTHH}
There is an isomorphism of graded left $\ca^{\vee}$-comodules 
\[ 
H_*(\THH(y(n))) \cong H_*(y(n)) \otimes E(\sigma \bxi_1, \sigma \bxi_2,\ldots \sigma \bxi_n),
\]
$|\sigma \bxi_i| = 2^i$, where the coaction 
\[ 
\nu_n\co  H_*(\THH(y(n)))\to \ca^{\vee} \otimes H_*(\THH(y(n)))
\]
on elements~$x\in H_*(y(n))$ is determined by the restriction of the coproduct on $\ca^{\vee}$ to $H_*(y(n))\subset \ca^{\vee}$, the coaction on $\sigma \bxi_i$ is determined by the formula~$\nu_n(\sigma \bxi_i)=(1\otimes \sigma )\nu_n(\bxi_i)$, and the coaction on classes of the form $xy$ is determined by $\nu_n(xy) = \nu_n(x)\nu_n(y)$. 
\end{prop}

\begin{proof}
The $\mathrm{E}_2$-page of the B{\"o}kstedt spectral sequence
\[
E^{*,*}_2 \cong \HH_*(H_*(y(n))) \cong P(\bxi_1,\ldots,\bxi_n) \otimes E(\sigma \bxi_1,\ldots,\sigma \bxi_n)
\]
 maps injectively to the $\mathrm{E}_2$-page of the B\"okstedt spectral sequence for $H\f_2$. The latter spectral sequence is multiplicative and all the algebra generators are concentrated in B\"okstedt filtration zero and one. Consequently, the B\"okstedt spectral sequence for $H\f_2$ collapses and the injective map of spectral sequences implies that the B\"okstedt spectral sequence for $y(n)$ also collapses.

The B\"okstedt spectral sequence is a spectral sequence of $\ca^{\vee}$-comodules and the formula~$\nu_n(\sigma x)=(1\otimes \sigma )\nu_n(x)$ holds because the operator $\sigma$ is induced by a map of spectra~$\bT\wedge R\to \THH(R)$ (see e.g.~\cite[Eq.~5.11]{AR05}) and this determines the $\ca^{\vee}$-coaction modulo lower B\"okstedt filtration. Here $\bT\subset \mathbb{C}$ denotes the circle regarded as a topological group with the subspace topology.
\end{proof}

We will use the fact that the B{\"o}kstedt spectral sequence computing $H_*(\THH(y(n)))$ agrees with the B\"okstedt spectral sequence computing $H_*(\THH(H\f_2))$ up until degree~$2^{n+1}-2 = |\bxi_{n+1}|-1$. We will also frequently use the map 
\[ 
\phi_n\co \THH(y(n)) \longrightarrow \THH(H\f_2)
\] 
induced by the map~$y(n) \to H\mathbb{F}_2$. The rest of this section is dedicated to studying the induced map on homology 
\[ 
(\phi_n)_* \co H_*(\THH(y(n))) \to H_*(\THH(H\f_2)) \,. 
\] 

\begin{rem}\label{Rmk:SigmaDerivation}
The spectrum~$\THH(y(n))$ has a canonical circle action~$\bT_+ \wedge \THH(y(n)) \to \THH(y(n))$ compatible with a structure map~$\sigma \co \bT \wedge  y(n)\to \THH(y(n))$. Though the spectrum~$\THH(y(n))$ is not a ring spectrum, we can refer to `products' in $H_*(\THH(y(n))$ since the classes in $H_*(\THH(y(n)))$ are named by their image in the ring~$H_*(\THH(H\f_2))$ (cf.~Remark~\ref{Rmk:IntermediaryFiltration}). With this convention, the map~$\sigma$  acts as though it were a derivation on $H_*(\THH(y(n)))$ as in~\cite[Prop.~3.2]{MS93}. 
Indeed, the proof of~\cite[Prop.~3.2]{MS93} only relies on $R$ being an $\mathbb{E}_1$~ring spectrum. 
This behavior will be important for our analysis in the next section since the structure map~$\sigma$ determines the $d^2$-differentials in the homological $\bT$-Tate spectral sequence. 

If $x \in H_*(\THH(y(n)))$ satisfies $\sigma(x) = 0$, we will refer to $x$ as a \emph{$\sigma$-cycle}. If $x = \sigma(y)$ for some $y \in H_*(\THH(y(n)))$, we will refer to $x$ as a \emph{$\sigma$-boundary}. 
\end{rem}

We are now ready to prove the main result of this section, Proposition~\ref{Prop:ymapF}. Before that, we include the following example to illustrate some subtleties in understanding $(\phi_n)_*$. 

\begin{exm}\label{Examplenequals1}
We will fully describe the map 
\[ 
P(\bxi_1) \otimes E(\sigma \bxi_1) \cong H_*(\THH(y(1))) \overset{(\phi_1)_*}{\longrightarrow} H_*(\THH(H\f_2)) \cong P(\bxi_1,\bxi_2,\ldots) \otimes P(\sigma \bxi_1)
\]
induced by the map $\phi_1\co \THH( y(1)) \to \THH(H\f_2)$ as a map of $E(\sigma)$-modules in the category of $\ca^{\vee}$-comodules. We first describe this map as map of $E(\sigma)$-modules. 

By Remark~\ref{Rmk:IntermediaryFiltration}, we have $(\phi_1)_*(\bxi^i_1) = \bxi^i_1$ since $\bxi_1^i \in H_i(\THH(y(1)))$ has B{\"o}kstedt filtration zero. As noted already, the map $(\phi_1)_*$ sends classes in $H_*(\THH(y(n)))$ in B{\"o}kstedt filtration zero to the classes with the same name in $H_*(\THH(H\f_2))$. 

Moving on to B{\"o}kstedt filtration one, we know that either
\[ 
(\phi_1)_*(\sigma \bxi_1) = \sigma \bxi_1 \text{ or }(\phi_1)_*(\sigma \bxi_1) = \sigma \bxi_1 + \bxi_1^2
\] 
for degree reasons. If the latter formula holds, we may simply change our basis for the vector space~$H_2(\THH(y(1)))\cong \mathbb{F}_2\{\sigma \bxi_1,\bxi_1^2\}$ to account for this, so we may assume the former. 

We now analyze the key case. Consider the class~$\bxi_1 \sigma \bxi_1 \in H_3(\THH(y(1)))$. We claim that $(\phi_1)_*(\bxi_1 \sigma \bxi_1) \neq \bxi_1 \sigma \bxi_1$. In fact, we know that $\sigma (\bxi_1 \sigma \bxi_1)=0$ in $H_*(\THH(y(1)))$ and therefore 
$\bxi_1\sigma \bxi_1$ must map to a $\sigma$-cycle in $H_*(\THH(H\f_2))$. By Remark~\ref{Rmk:IntermediaryFiltration}, we know that $\bxi_1\sigma \bxi_1$ maps to the class of the same name modulo classes in lower B\"okstedt filtration. We also know that in $H_*(\THH(H\f_2))$, the equality $\sigma (\bxi_1\sigma \bxi_1)=\sigma \bxi_2$ holds. Therefore, there is an equality~$(\phi_1)_*(\bxi_1\sigma \bxi_1)=\bxi_1\sigma\bxi_1+y$ where $y$ is in B\"okstedt filtration zero and there is an equality~$\sigma y=\sigma\bxi_2$. The only such element in $H_*(\THH(H\f_2))$ with these properties is $\bxi_2$ itself. Thus, 
\[ 
(\phi_1)_*(\bxi_1\sigma \bxi_1)=\bxi_1\sigma\bxi_1+\bxi_2 \,.
\]
We then claim that $(\phi_1)_*(\bxi_1^{2k}\sigma \bxi_1)=\bxi_1^{2k}\sigma \bxi_1$. We know that $\sigma (\bxi_1^{2k}\sigma \bxi_1)=0$ in both the source and target. Therefore, the only possibility is that we add $\sigma$-cycles in either the source or target of the map. Since this does does not affect the map up to isomorphism of $E(\sigma)$-modules, we may assume $(\phi_1)_*(\bxi_1^{2k}\sigma \bxi_1)=\bxi_1^{2k}\sigma \bxi_1$. 

We also claim that $(\phi_1)_*(\bxi_1^{2k+1}\sigma \bxi_1)=\bxi_1^{2k+1}\sigma \bxi_1+\bxi_1^{2k}\bxi_2$. Again, we know $\sigma (\bxi_1^{2k+1}\sigma \bxi_1)=0$ in $H_*(\THH(y(1)))$ whereas $\bxi_1^{2k+1}\bxi_1=\bxi_1^{2k}\sigma \bxi_2$ in $H_*(\THH(H\f_2))$. Therefore, we must add a term~$y$ in B\"okstedt filtration zero such that $\sigma y= \bxi_1^{2k}\sigma \bxi_2$ and the only possibility is $\bxi_1^{2k} \bxi_2$. This completely determines the map up to isomorphism of $E(\sigma)$-modules.

We now describe the map as a map of $E(\sigma)$-modules in the category of $\ca^{\vee}$-comodules up to some indeterminacy. First, note that there are no $\sigma$-cycles in the degree of $\bxi_1^{2k+1}\sigma\bxi_1$ in lower B{\"o}kstedt filtration and thus we know the answer for $\bxi_1^{2k+1}\sigma \bxi_1$ completely as a map of $\ca^{\vee}$-comodules. This also forces the $\ca^{\vee}$-comodule structure on these elements. For example, since 
\[
	\begin{array}{rcl}
		\nu_{\infty}(\bxi_1^{2k+1}\sigma\bxi_1+\bxi_1^{2k}\bxi_2)&=&(\bxi_1\otimes 1+1\otimes \bxi_1)^{2k+1}(1\otimes \sigma \bxi_1)+ \\
		&&(\bxi_1^{2k}\otimes 1+1\otimes \bxi_1^{2k})(\bxi_2\otimes 1 + \bxi_1\otimes \bxi_1^2+1\otimes \bxi_2) \\
		&=& (\bxi_1\otimes 1+1\otimes \bxi_1)^{2k+1}(1\otimes \sigma \bxi_1)+ \bxi_1^{2k}\bxi_2\otimes 1+ \bxi_2\otimes \bxi_1^{2k} + \\
		&& \bxi_1^{2k+1}\otimes \bxi_1^2+\bxi_1\otimes \bxi_{1}^{2k+2}+ \bxi_1^{2k}\otimes \bxi_2+1\otimes \bxi_1^{2k}\bxi_2 ,\\
	\end{array}
\]
we know that 
\[  
\begin{array}{rcl}
\nu_1(\bxi_1^{2k+1}\sigma \bxi_1)&=&(\bxi_1\otimes 1+1\otimes \bxi_1)^{2k+1}(1\otimes \sigma \bxi_1)+ \bxi_1^{2k}\bxi_2\otimes 1+  \bxi_2\otimes \bxi_1^{2k} + \\
&&\bxi_1^{2k+1}\otimes \bxi_1^2+\bxi_1\otimes \bxi_{1}^{2k+2}+ \bxi_1^{2k}\otimes \bxi_2+1\otimes \bxi_1^{2k}\bxi_2  \,.
\end{array}
\]
In the case of $\bxi_1^{2k}\sigma \bxi_1$, adding $\sigma$-cycles of the form~$\bxi_1^{j}$ does not affect the comodule structure on the source up to a change of basis, since $\bxi_1^{j}$ is also in the target. However, if we add a $\sigma$-cycle in $H_*(\THH(H\f_2))$ that is not in the source, then this affects the comodule structure on the source. We therefore determine $\phi_*$ up to this indeterminacy. In summary, $\bxi_1^{2k}\sigma \bxi_1 $ maps to $\bxi_1^{2k}\sigma \bxi_1$ up to $\sigma$-cycles in $H_*(\THH(H\f_2))$ that are not in the image of $H_*(\THH(y(1)))$. In Proposition~\ref{Prop:ymapF}, we will describe $(\phi_n)_*$ up to the same type of indeterminacy.
\end{exm}

\begin{rem}
In fact, we can actually avoid indeterminacy in the previous example because $\bxi_1^{2k}\sigma \bxi_1$ is a $\sigma$-boundary. Since we know that $\bxi_1^{2k+1}$ maps to the element of the same name, we see that $\bxi_1^{2k}\sigma \bxi_k$ must map to the element of the same name without any indeterminacy. This argument no longer applies when studying $(\phi_n)_*$ for $n \geq 2$ since there will typically be additional elements in lower B{\"o}kstedt filtration. 
\end{rem}

We may choose a basis of $H_*(\THH(y(n)))$ so that $\sigma$ behaves as a derivation at the level of symbols, i.e. there is an equality up to lower B{\"o}kstedt filtration~$\sigma(xy) = \sigma(x)y + x\sigma(y)$ for $x,y \in H_*(y(n))$. Indeed, we may apply~\cite[Prop.~3.2]{MS93} to see that the class $\sigma(xy)$ is detected by $\sigma_*(xy)$ in the $\mathrm{E}_2$-page of the B{\"o}kstedt spectral sequence (where $\sigma_* \colon \thinspace H_*(y(n)) \to \HH_*(H_*(y(n)))$ is defined by $z \mapsto 1 \otimes z$). We have $\sigma_*(xy) = \sigma_*(x)y + x \sigma_*(y)$ since $\sigma_*$ is a derivation in Hochschild homology of the graded commutative ring~$H_*(y(n))$~\cite[p.~7]{MS93}, and~$\sigma_*(x)y + x\sigma_*(y)$ detects the element we call $\sigma(x)y + x\sigma(y)$ in $H_*(\THH(y(n)))$. 

The coaction on $\sigma \bxi_k$ is given by 
\[ 
\nu_n(\sigma \bxi_k)=(1\otimes\sigma ) \left ( \sum_{i+j=k} \bxi_i\otimes \bxi_j^{2^i} \right ) \,.
\]
Since $\sigma$ is a derivation (in the sense of Remark~\ref{Rmk:SigmaDerivation}), we see that the only term that is nontrivial in the formula for $\nu_n(\sigma \bxi_k)$ is $1\otimes \sigma \bxi_k$. 
We therefore conclude that $\sigma \bxi_k$ is a comodule primitive for all $k$. 

We now proceed to the main result of this section. We thank Vigleik Angeltveit for discussions which led to a simplification of the proof of Proposition~\ref{Prop:ymapF}; we use some notation from~\cite[Prop.~4.12]{Ang08}. We also note once and for all that elements in $H_*(\THH(y(n)))$ are only well-defined up to lower B\"okstedt filtration as in~\cite[Sec. 5]{Ang08}. 

\begin{defin}\label{compsigmacycles}
Let $\bfilt(x)$ be the B\"okstedt filtration of an element $x$. Define
\[ J_n = ( x \in H_*(\THH(H\f_2)) \setminus \im (\phi_n)_* : \bfilt(x)\le n \text{ and } \sigma(x)=0 )\]
to be the ideal generated by all~$\sigma$-cycles $x$ in~$H_*(\THH(H\f_2))$ in B\"okstedt filtration less than or equal to $n$ that are not in the image of $(\phi_n)_*$. We refer to these elements as the \emph{complementary~$\sigma$-cycles} in the proof of the following proposition. 
\end{defin}

\begin{prop}\label{Prop:ymapF}
Let $x_i=\sigma \bxi_i\sigma \bxi_{i+1}\ldots \sigma \bxi_n$. The map 
\[ 
(\phi_n)_*\co H_*(\THH(y(n)))\to H_*(\THH(H\f_2)) 
\]
is determined modulo intermediary filtration and complementary $\sigma$-cycles by the following:

\begin{enumerate}[(a)]

\item \label{item 1 Prop:ymapF} We may find a representative of $\bxi_i x_i$ so that
\[
 (\phi_n)_*(\bxi_ix_i) =\bxi_ix_i+\bxi_{n+1}
\]
modulo intermediary filtration and complementary $\sigma$-cycles. 

\item  \label{item 2 Prop:ymapF}  We may choose representatives for $y\in H_*(y(n))$ or $y\in E(\sigma \bxi_1,\sigma \bxi_2,\ldots ,\sigma \bxi_n)$ such that
\[
 (\phi_n)_*(y)=y \,.
\]
modulo intermediary filtration and complementary $\sigma$-cycles. That is, we can find representatives so that $f_0: H_*(\THH(y(n))) \to \ca^{\vee}$ sends $y$ to zero modulo complementary $\sigma$-cycles. 

\item  \label{item 3 Prop:ymapF}  For all remaining products in $H_*(\THH(y(n)))$, we may choose representatives so the map $(\phi_n)_*$ is multiplicative modulo intermediary filtration and complementary $\sigma$-cycles \,.

\end{enumerate}
\end{prop}

\begin{proof}
We begin with the proof of Item~\ref{item 1 Prop:ymapF}. We will proceed by downward induction on $i$, starting with the case $i=n$. Observe that in $H_*(\THH(H\f_2))$, we have $\sigma (\bxi_ix_i)=\sigma \bxi_{n+1}$ for all $i \leq n$. Since $\sigma \bxi_{n+1} \notin H_*(\THH(y(n)))$, we must have that $\bxi_n\sigma \bxi_n \in H_*(\THH(y(n)))$ maps to $\bxi_n\sigma \bxi_n+z \in H_*(\THH(H\f_2))$ with $z \neq 0$ some element in lower B\"okstedt filtration such that $\sigma (\bxi_ix_i+z)=0$. 
In particular, $\sigma (z)=\sigma \bxi_{n+1}$. The only element~$z \in H_*(\THH(H\f_2))$ in lower B\"okstedt filtration such that $\sigma z= \sigma \bxi_{n+1}$ is $\bxi_{n+1}$.  

Suppose now that Item~\ref{item 1 Prop:ymapF} holds for all $n \geq j > i$, i.e.~$\bxi_jx_j$ maps to $\bxi_jx_j+\bxi_{n+1}$ modulo intermediary filtration for all $j>i$. By the same argument as above, the class $\bxi_ix_i$ maps to $\bxi_ix_i +z$ where $z \neq 0$ is some element in lower filtration such that $\sigma(\bxi_i x_i + z) = 0$. Examining $H_*(\THH(H\f_2))$, we have that either $z =\bxi_{n+1}$ or 
\[ 
z\in \{\bxi_nx_n , \bxi_{n-1}x_{n-1}, \ldots, \bxi_{i+1}x_{i+1}\} 
\]
up to $\sigma$-cycles in $H_*(\THH(y(n)))$. If the former holds, then we are done. If the latter holds, then we can add $z$ to the source and we know that $\bxi_ix_i+z$ maps to $\bxi_ix_i+\bxi_{n+1}$, modulo intermediary filtration, by the inductive hypothesis. This completes the proof of \ref{item 1 Prop:ymapF}. 

We now turn to Item~\ref{item 2 Prop:ymapF}. For $y\in H_*(y(n))$ it is clear that $(\phi_n)_*(y)=y$ because all such $y$ are in B\"okstedt filtration zero. For $y\in E(\sigma \bxi_1,\sigma \bxi_2,\ldots ,\sigma \bxi_n)$, we know that after a possible change of basis, each of these elements~$y$ is a comodule primitive in $H_*(\THH(y(n)))$ and therefore each such $y$ maps to the element of the same name in $H_*(\THH(H\f_2))$. 

We now prove Item~\ref{item 3 Prop:ymapF}. By our naming conventions, a product~$x \cdot y \in H_*(\THH(y(n)))$ maps to $x \cdot y + z \in H_*(\THH(H\f_2))$, where $z$ is some (possibly trivial) element in lower B{\"o}kstedt filtration. Since we are only concerned with describing $z$ modulo intermediary filtration and complementary $\sigma$-cycles, it suffices to describe $f_0(x \cdot y) \in \ca^{\vee}$ modulo complementary $\sigma$-cycles. Since there is no indeterminacy in filtration zero, we have $f_0(x \cdot y) = f_0(x) \cdot f_0(y)$ modulo complementary $\sigma$-cycles. Thus, $(\phi_n)_*(x \cdot y) = x \cdot y$ modulo intermediary filtration and complementary $\sigma$-cycles, as claimed.
\end{proof}
We now note that $(\phi_n)_*$ is also a map of $E(\sigma)$-modules in $\ca^{\vee}$-comodules. Since $(\phi_n)_*$ is exotic in some cases, there is an exotic $\ca^{\vee}$-coaction on some elements in $H_*(\THH(y(n)))$.

\begin{cor}\label{cor:coaction}
Modulo intermediary filtration and $\sigma$-cycles in $H_{*}(\THH(y(n)))$, the $\ce^{\vee}$-coaction on $H_*(\THH(y(n)))$ is determined by the formula
\begin{equation}\label{exoticcoaction} 
\nu_n(\bxi_ix_i)= \sum_{j=0}^{i}\bxi_{j}\otimes \bxi_{i-j}^{2^j}x_i + \sum_{j=1}^{n+1} \bxi_j \otimes \bxi_{n+1-j}^{2^{n+1}}  
\end{equation}
for  $x_i=\sigma \bxi_i\sigma \bxi_{i+1}\ldots \sigma \bxi_n$ and $1\le i\le n$, the usual coaction on the $\ce^{\vee}$-comodule $H_*(y(n))$, and primitivity of the coaction on the $\ce^{\vee}$-comodule $E(\sigma \bxi_1,\ldots ,\sigma \bxi_n)$. 
\end{cor}

\begin{proof}
By Proposition~\ref{Prop:ymapF}, this follows from the commutative diagram 
\[
\begin{tikzcd}
H_*(\THH(y(n))) \arrow{r}{\nu_{n}} \arrow{d}{(\phi_{n})_{*}} & \ca^{\vee} \otimes H_*(\THH(y(n)))  \arrow{d}{\id \otimes (\phi_{n})_{*}}\\
H_*(\THH(\f_{2}))\arrow{r}{\nu_{\infty}}  &  \ca^{\vee} \otimes H_*(\THH(\f_{2}))
\end{tikzcd}
\] 
and the coaction of the dual Steenrod algebra on $H_*(\THH(\f_{2}))$, see~\cite[Thm.~5.12(3)]{AR05}. 
The fact that elements in $E(\sigma \bxi_1,\ldots ,\sigma \bxi_n)$ are primitive as $\ce^{\vee}$-comodules follows because $E(\sigma \bxi_1,\ldots ,\sigma \bxi_n)$ is concentrated in even degrees. 
\end{proof}

\section{Continuous mod $2$ homology of $\TP(y(n))$}\label{Sec:TP}
In many classical trace methods computations, topological periodic cyclic homology is understood using the homotopical Tate spectral sequence described by Greenlees--May~\cite{GM95}. In Section~\ref{Sec:Limitations}, we explain why this method of understanding $\TP(R)$ is not tractable when $R=y(n)$ for $n<\infty$. 
In Section~\ref{homTate}, we apply an alternative approach to understanding $\TP(R)$ inspired by foundational work of Bruner--Rognes~\cite{BR05}, the \emph{homological} Tate spectral sequence. We analyze this spectral sequence to compute the \emph{continuous} homology $H^c_*(\TP(y(n)))$ in Proposition~\ref{tatebasis}. 

\subsection{Limitations of the homotopical Tate spectral sequence}\label{Sec:Limitations}
Let $R$ be an $\mathbb{E}_1$~ring spectrum. The topological periodic cyclic homology spectrum~$\TP(R)$ arises in many classical trace methods computations, and it is now part of the definition of topological cyclic homology as the fiber 
\[ 
\mathrm{TC}(R)_{2} \to  \mathrm{TC}^{-}(R)_{2} \overset{\can-\varphi}{\longrightarrow} \TP(R)_{2}
\]
after work of~\cite{NS18}. For example, when $p$ is an odd prime, the spectrum~$\TP(H\f_p)$ appears in Hesselholt and Madsen's computation of the algebraic K-theory of finite algebras over the Witt vectors of perfect fields~\cite{HM97}. 
Similarly, it plays an important role in the computation of $\TC(\z_2)/2$ by Rognes~\cite{Rog99}. In both cases, they analyze the mod $p$ homotopical Tate spectral sequence
\[\widehat{E}^2 = \widehat{H}^{-*}(\bT; \pi_*(\THH(R)/p)) \Longrightarrow \pi_*(\TP(R)/p)\]
defined in~\cite{GM95}. We will review the filtration used to define this spectral sequence when we define the homological Tate spectral sequence in Subsection~\ref{homTate}.

When $R=y(\infty) = H\f_2$, this spectral sequence is fairly simple. By B\"okstedt periodicity~\cite{BokZ}, there is an isomorphism $\pi_*(\THH(H\f_2))\cong P(u)$ with $|u|=2$, so one has a familiar checkerboard pattern on the $\mathrm{E}^2$-page and the spectral sequence collapses. On the other hand, when $R = y(n)$ for $n < \infty$, this spectral sequence appears to be intractable.

\begin{exm} 
We have $y(0) = \mathbb{S}$ and $\THH(\mathbb{S}) \simeq \mathbb{S}$ as $\mathbb{T}$-spectra. There is an equivalence of spectra 
\[ 
\TP(\mathbb{S}) \simeq \Sigma^2\mathbb{C}P^{\infty}_{-\infty}
 \]
by~\cite[Thm. 16.1]{GM95}. The homotopy groups of $\mathbb{C}P^{\infty}_{-\infty}$ are less well understood than the stable homotopy groups of spheres (cf.~\cite{Rog02}).
\end{exm}

Moreover,~\cite[Thm. 1]{BCS10} implies that
\[ 
\THH(y(n))\simeq \Th(L^{\eta}(Bf)) 
\]
where $f:\Omega J_{2^n-1}(S^2)\to B\GL_1\mathbb{S}$ is the map defining $y(n)$ as $\Th(f)=y(n)$ and~$\Th(L^{\eta}(Bf))$ is the Thom spectrum of the composite map~$L^{\eta}(Bf)$ defined as
\[ 
LB\Omega J_{2^n-1}(S^2)\overset{L(Bf)}{\longrightarrow} LB^2F\simeq B\GL_1\mathbb{S}\times B^2\GL_1\mathbb{S}\overset{B\GL_1\mathbb{S}\times \eta }{\longrightarrow} B\GL_1\mathbb{S} \times B\GL_1\mathbb{S} \to B\GL_1\mathbb{S} \,.
\]
This spectrum has homotopy groups at least as complicated as $\pi_*(y(n))$, which are only known in a finite range. Since we want to understand large-scale phenomena in these homotopy groups, we will adopt a different approach.

\subsection{Homological Tate spectral sequence for $\THH(y(n))$}\label{homTate}
In notes from a talk by Rognes~\cite{Rog11}, it is shown using the homological homotopy fixed point spectral sequence and the inverse limit Adams spectral sequence~\cite{LDMA80} that there is an isomorphism of graded abelian groups
\[ 
\pi_*(\TC^{-}(H\f_2)) \cong \prod_{i \in \z} \Sigma^{2i} \z_2 \,.
\]
We recall this calculation in Proposition~\ref{Prop:Rog}. A similar argument shows that there is an isomorphism of graded abelian groups
\[ 
\pi_*(\TP(H\f_2)) \cong \prod_{i \in \z} \Sigma^{2i} \z_2 \,.
\]

\begin{defin}[Homological Tate spectral sequence \cite{BR05}]\label{Greenleesfilt}
Let $R$ be an $\mathbb{E}_1$~ring spectrum. The \emph{homological Tate spectral sequence} has the form 
\[
\widehat{\mathrm{E}}^{2}_{*,*}= \widehat{H}^{-*}(\bT;H_*(\THH(R))) \Longrightarrow H^c_*(\TP(R)) \,.
\]
It arises from the Greenlees filtration of $\TP(R) = \THH(R)^{t\bT} = [F(E{\bT}_+,\THH(R)) \wedge \widetilde{E\bT}]^{\bT}$ defined by setting (cf.~\cite[Sec.~2]{BR05})
\[
\TP(R)[i] := [F(E\bT_+,\THH(R)) \wedge \widetilde{E\bT}/\widetilde{E\bT}_i ]^{\bT}
\]
where $\widetilde{E\bT}_i$ is the cofiber of the map~$E\bT^{(i)}_+\to S^0$, where $E\bT^{(i)}$ is the $i$-th skeleton of $E\bT$, if $i\ge 0$ and~$\widetilde{E\bT}_i$ is the Spanier--Whitehead dual of $\widetilde{E\bT}_{-i-1}$ if $i<0$~\cite[p.~437]{Gre87}.
The limit
\[
H^c_*(\TP(R)) := \underset{i}{\lim \thinspace}H_*(\TP(R)[i])
\]
is called the \emph{continuous homology} of $\TP(R)$. For $0 \leq n \leq \infty$, we will denote the $\mathrm{E}^r$-page of the homological Tate spectral sequence converging to $H^c_*(\TP(y(n)))$ by $\widehat{\mathrm{E}}^{r}(n)$. Note that there is also an eventually
constant filtration of $\TP(R)[i]$ defined by 
\[ \mathrm{fil}_{\mathrm{Gre}}^{j}\TP(R)[i]=\begin{cases}
										\TP(R)[i] \text{ if } j< i\\ 	
										\TP(R)[j] \text{ if } j\ge i\\ 
\end{cases}
\]
whose associated spectral sequence we call the \emph{approximate homological Tate spectral sequence}. We write $\mathrm{Fil}_{\mathrm{Gre}}^jH_*(\TP(R)[i]))=H_*(\mathrm{fil}_{\mathrm{Gre}}^j\TP(R)[i])$ for the associated filtration of $H_*(\TP(R)[i])$ and we write 
\[
\Gr_{\mathrm{Gre}}^jH_*(\TP(R)[i])=\mathrm{Fil}_{\mathrm{Gre}}^jH_*(\TP(R)[i]))/\mathrm{Fil}_{\mathrm{Gre}}^{j+1}H_*(\TP(R)[i]))\,.
\] 
\end{defin}

\begin{lem}\label{Lem:E2Hat}
There is an additive isomorphism
\[
\widehat{\mathrm{E}}^{2}(n) \cong P(t,t^{-1}) \otimes H_*(\THH(y(n))) \cong P(t,t^{-1}) \otimes P(\bxi_1,\bxi_2,\ldots,\bxi_n) \otimes E(\sigma \bxi_1,\ldots, \sigma \bxi_n)
\]
where $|t| = (-2,0)$, $|\bxi_i| = (0,2^i-1)$, and $|\sigma \bxi_i| = (0,2^i)$. 
\end{lem}
In~\cite[Prop.~3.2]{BR05}, Bruner and Rognes show that $d^2(x) = t \cdot \sigma(x)$ in the homological Tate spectral sequence. Therefore, in order to compute $\widehat{\mathrm{E}}^3(n)$, we need to understand the $\bT$-action on $H_*(\THH(y(n)))$. This can be understood using Proposition~\ref{Prop:ymapF} and the relation~$(\phi_n)_*(\sigma(x)) = \sigma((\phi_n)_*(x))$ which follows from naturality of $\sigma$. 

\begin{defin}\cite[Prop.~6.1.(a)]{BR05}
Let $k \geq 1$. Define $\bxi_{k+1}' \in H_*(\THH(H\f_2))$ by 
\[
\bxi_{k+1}' := \bxi_{k+1} + \bxi_k \sigma \bxi_k \,.
\]
\end{defin}

\begin{prop}\label{tatebasis}
There is an isomorphism of graded $\f_2$-vector spaces 
\[
H^c_*(\TP(y(n))) \cong P(t,t^{-1}) \otimes P(\bxi^2_1,\bxi_2',\ldots,\bxi_{n}') \otimes E(\bxi_n \sigma \bxi_n)
\]
with $|t| = (-2,0)$, $|\bxi_i| = (0,2^i-1)$, and $|\sigma \bxi_i| = (0,2^i)$.
\end{prop}
\begin{proof}
First, note that $\widehat{\mathrm{E}}^2_{**}(n)$ was computed in Lemma~\ref{Lem:E2Hat}. 
The homological Tate spectral sequence is not a multiplicative spectral sequence since $\THH(y(n))$ is not a ring spectrum, but it is a module over the spectral sequence for the sphere~$\{\widehat{\mathrm{E}}_{*,*}^r(0)\}_r$. Consequently, the equality $d^r(t)=0$ holds and the differentials are $t$-linear. We have differentials $d^2(\bxi_k) = t \sigma\bxi_k$ and thus $d^2(t^m \bxi_k) = t^{m+1} \sigma\bxi_k$ for $m \in \z$ by $t$-linearity.

Recall that $x_i=\sigma\bxi_i\cdots\sigma\bxi_n$. Any class of the form $y \bxi_i x_i$, where $y$ is a $\sigma$-cycle, is a $d^2$-cycle in the homological Tate spectral sequence converging to $H^c_*(\TP(y(n)))$. Many of these classes are also $d^2$-homologous; in particular, 
\[
d^2(y \bxi_i \bxi_n \sigma \bxi_i\cdots \sigma \bxi_{n-1}) = ty \bxi_i x_i + ty \bxi_n x_n \,.
\]
 
Using these relations and the fact that this spectral sequence is a module over the spectral sequence for the sphere, we obtain an additive isomorphism (cf.~\cite[Prop.~6.1]{BR05})
\[
\widehat{\mathrm{E}}^3_{*,*}(n) \cong P(t,t^{-1}) \otimes P(\bxi^2_1, \bxi_2', \bxi_3', \ldots) \otimes E(\bxi_n \sigma \bxi_n) \,.
\]

To see that there are no further differentials, we use the map of spectral sequences induced by the $\bT$-equivariant map~$\THH(y(n)) \to \THH(H\f_2)$. 
The homological Tate spectral sequence converging to $H^c_*(\TP(H\f_2))$ has $\widehat{\mathrm{E}}^3$-page 
\[
\widehat{\mathrm{E}}^3_{*,*}(\infty) \cong P(t,t^{-1})\otimes P(\bxi^2_1,\bxi_2', \bxi_3', \ldots) \,.
\]
All of the generators are permanent cycles by \cite[Thm.~5.1]{BR05}, so there are no further differentials. The map $\widehat{\mathrm{E}}_{*,*}^3(n)\to \widehat{\mathrm{E}}_{*,*}^3(\infty)$ is injective by Proposition~\ref{Prop:ymapF} so we can conclude that there is also an isomorphism~$\widehat{\mathrm{E}}^3(n) \cong  \widehat{\mathrm{E}}^\infty(n)$. 
\end{proof}

A similar proof can be used to compute the homology of the spectra~$\TP(y(n))[i]$ which were used to define the filtration of $\TP(y(n))$ giving rise to the homological Tate spectral sequence. Indeed, one may truncate the homological Tate spectral sequence to obtain a spectral sequence which converges strongly to $H_*(\TP(y(n))[i])$. 

\begin{notation}\label{notation L(i)}
When computing the truncated homological Tate spectral sequence or the truncated homological homotopy fixed point spectral sequence, we denote the left-most column (using Serre grading) by
\[ 
\mathrm{L}(i):= \left ( H_*(\THH(y(n))/\im (d^{2i-2,*}_2 ) \right )\{t^i\} 
\] 
where the integer~$n$ is understood from the context.
\end{notation}

\begin{cor}
There is an isomorphism of graded~$\f_2$-vector spaces
\begin{align*}
H_*(\TP(y(n))[i]) \cong & \left[P(t^{-1})\{t^{i-1}\} \otimes P(\bxi^2_1, \bxi_2', \ldots, \bxi_n') \otimes E(\bxi_n \sigma \bxi_n) \right] \oplus \mathrm{L}(i) 
\end{align*}
with $|t| = (-2,0)$, $|\bxi_i| = (0,2^i-1)$, $|\sigma \bxi_i| = (0,2^i)$, and $P(t^{-1})\{t^{i-1}\}$ is viewed as a $P(t^{-1})$-submodule of $P(t,t^{-1})$. 
\end{cor}

If $X = \lim_i X_i$ is the homotopy limit of bounded below spectra~$X_i$ of finite type, then the inverse limit Adams spectral sequence
\[
\mathrm{E}_2^{*,*} = \Ext^{*,*}_{\ca^{\vee}}(\f_2,H^c_*(X)) \Longrightarrow \pi_*(X)
\]
arises from the filtration of $X$ obtained by taking the inverse limit of compatible Adams filtrations of the spectra $X_i$, where the left-hand side is computed using the continuous $\ca^{\vee}$-coaction on $H^c_*(X)$. For details, see~\cite[Sec.~2]{LNR12}. Taking $X = \TP(y(n))$ gives a method for calculating $\pi_*(\TP(y(n)))$. In view of Rognes' computation of $\pi_*(\TC^{-}(H\f_2))$ \cite{Rog11}, one might suspect that the inverse limit Adams spectral sequence could be used to compute the homotopy groups~$\pi_*(\TP(y(n)))$ directly.  
However, this approach is significantly less tractable for $n < \infty$ since $\ca^{\vee}$ coacts nontrivially on $P(t,t^{-1}) \subset H^c_*(\TP(y(n)))$. ~This problem is avoided when $n=\infty$ as follows. There is an $\ca^{\vee}$-comodule isomorphism
\[
H^c_*(\TP(H\f_2)) \cong P(t,t^{-1}) \otimes H_*(H\z_2) \cong P(t,t^{-1}) \otimes (\ca//E(0))_* \,.
\]
A change-of-rings isomorphism then gives
\[
\mathrm{E}_2^{*,*} \cong \Ext_{E(\bxi_1)}^{*,*}(\f_2,P(t,t^{-1})) \,.
\]
Since $\bxi_1$ is in an odd degree and $P(t,t^{-1})$ is concentrated in even degrees, the $E(\bxi_1)$-coaction on $P(t,t^{-1})$ is trivial. 
Therefore
\[
\mathrm{E}_2^{*,*} \cong P(t,t^{-1}) \otimes \Ext_{E(\bxi_1)}^{*,*}(\f_2,\f_2)
\]
and the spectral sequence collapses for degree reasons. 
The key simplification in the next section is that we can replace the functor~$\Ext_{\ca^{\vee}}^{*,*}(\f_2,-)$ by the functor~$\Ext_{E(\bxi_n)}^{*,*}(\f_2,-)$ if we compute connective Morava K-theory instead of stable homotopy because of the change of rings isomorphism.

\subsection{Interpretation in terms of $z(n)/v_{n}$}\label{Sec:lilASS}
To determine the chromatic complexity of $\TP(y(n))$, we need to compute $K(m)_*(\TP(y(n)))$. We can identify a piece of the continuous homology~$H^c_*(\TP(y(n)))$ computed in Proposition~\ref{tatebasis} with the homology~$H_*(z(n)/v_n)$. 
Our identification may not hold as $\ca^{\vee}$-comodules due to the possibility of complementary $\sigma$-cycles in the image of $(\phi_n)_*$, but if we restrict to a smaller sub-Hopf-algebra of $\ca^{\vee}$, it does. Recall the definition of $\mathcal{E}^{\vee}$ from Section~\ref{conventions}.

\begin{prop}\label{yzrel}
There is an isomorphism of graded~$\mathbb{F}_{2}$-modules 
\begin{equation}\label{map yzrel}
H_*^c(\TP(y(n))) \cong P(t,t^{-1}) \otimes H_*(z(n)/v_n) \,.
\end{equation}
Moreover, the associated graded~$\Gr_{\mathrm{Gre}}^{*}(H_*^c(\TP(y(n))))$ of the Greenlees filtration on the continuous $\mathcal{E}^{\vee}$-comodule~$H_*^c(\TP(y(n)))$ can be identified with $P(t,t^{-1}) \otimes H_*(z(n)/v_n)$ as a continuous $\mathcal{E}^{\vee}$-comodule. 
\end{prop}

\begin{proof}
First, note that the $\ce^{\vee}$-coaction on $P(t,t^{-1})=H_{*}^{c}(\TP(\mathbb{S}))$ is trivial because $|t|=-2$ and each $\bxi_i$ is in an odd degree for all $i\ge 0$. We also know that $H_*^c(TP(y(n)))$ is a $H_*^c(TP(\mathbb{S})))$-module. 
The desired isomorphism is therefore given by a map
\[
P(\bxi^2_1,\ldots,\bxi^2_n) \otimes E(\bxi_2', \ldots, \bxi_n') \otimes E(\bxi_n \sigma \bxi_n) \to  H_*(z(n)/v_n)
\]
which is determined by 
\[
\bxi_i^2 \mapsto \bxi_i^2, \quad \bxi_j^{\prime} \mapsto \bxi_j, \quad \bxi_n \sigma \bxi_n \mapsto \bxi_{n+1}
\]
for $1 \leq i \leq n$ and $2 \leq j \leq n$. This is clearly an additive isomorphism, so it only remains to calculate the action of $Q_m$ for each $m \geq 0$. 

We wish to show that the $\mathcal{E}^\vee$-coaction on the elements in $\Gr_{\mathrm{Gre}}^{*}(H^c_*(\TP(y(n))))$ coincides with the $\mathcal{E}^\vee$-coaction on their images in $\Gr_{\mathrm{Gre}}^{*}H^c_*(\TP(\mathbb{F}_{2}))$. 
Since the coaction on 
\[\Gr_{\mathrm{Gre}}^{*}(H^c_*(\TP(\mathbb{F}_{2}))) \cong P(t,t^{-1})\otimes H_{*}(H\mathbb{Z}) \,.\]
is determined from a subquotient of the coaction on $\mathcal{A}^\vee$ using the module structure over the homological Tate spectral sequence for the sphere spectrum, we will compute the continuous $\ca^{\vee}$-coaction 
\[
\nu_n \co \Gr_{\mathrm{Gre}}^{*}H^c_*(\TP(y(n))) \to \ca^{\vee} \widehat{\otimes} \Gr_{\mathrm{Gre}}^{*}H^c_*(\TP(y(n)))
\]
by comparison with the known coaction
\[
\nu_\infty \co  \Gr_{\mathrm{Gre}}^{*}H^c_*(\TP(H\f_2)) \to \ca^{\vee} \widehat{\otimes} \Gr_{\mathrm{Gre}}^{*}H^c_*(\TP(H\f_2)) 
\]
using the map on continuous homology induced by the map 
\[
(\phi_n)_*: H_*(\THH(y(n))) \to H_*(\THH(H\f_2))
\]
studied in Proposition~\ref{Prop:ymapF}. Therefore we will observe that the coaction on $\Gr_{\mathrm{Gre}}^{*}H^c_*(\TP(y(n)))$ is determined from the coaction on $H_{*}(\THH(y(n))$ from Corollary~\ref{cor:coaction} after passing to a subquotient. 

For $x \in P(\bxi_1^2,\ldots,\bxi_n^2)$, we have $(\phi_n)_*(x) = x$ since $\bfilt(x) = 0$, so $\nu_n(x) = \nu_\infty(x)$. 

We calculate $\nu_n(\bxi_{i}')$ as follows. We have $(\phi_n)_*(\bxi_i') = \bxi_i' + y$ where $y$ is a $\sigma$-cycle with $\bfilt(y) = 0$. If $y \notin  J_n$, then we have $\nu_n(\bxi_i') = \nu_\infty(\bxi_i')$ modulo intermediary filtration. If $y \in J_n$, then $y$ is divisible by $\bxi_{n+r}^2$ or $\sigma \bxi_{n+r}$ for $r \geq 1$. In both cases, $\nu_\infty(y)$ does not contain any terms of the form $\bxi_m \otimes z$ for any $m$, so $Q_m$ acts trivially on $y$ for all $m$ and the action of $Q_i$ on $\bxi_m'$ is unaffected by $y$. Thus $\nu_n(\bxi_{i}')$ agrees with $\nu_{\infty}(\bxi_i)$. 

We have $(\phi_n)_*(\bxi_n \sigma \bxi_n) = \bxi_n \sigma \bxi_n + \bxi_{n+1}$ modulo intermediary filtration and complementary $\sigma$-cycles. Therefore, the equality 
\[ 
\nu_n(\bxi_n \sigma \xi_n) = \nu_\infty(\bxi_n \sigma \bxi_n) + \sum_{i=1}^{n+1}\bxi_i \otimes\bxi_{n+1-i}^{2^i} 
\] 
holds modulo indermediate filtration and complementary $\sigma$-cycles so as before we can argue that $Q_m(\bxi_n \sigma \bxi_n) = Q_m(\bxi_{n+1}')$ for all $m$.

Finally, we use Proposition~\ref{Prop:ymapF}(c) and the discussion above to determine the coaction on symbolic products of classes. 
\end{proof}
\begin{rem}
If the map~\eqref{map yzrel} can be upgraded to a map of $\mathcal{A}^{\vee}$-comodules induced by an $\mathbb{T}$-equivariant map $z(n)/v_{n}\to \THH(y(n))$, then this would imply an equivalence 
\[ 
((z(n)/v_{n})^{t\mathbb{T}})_{2}\to \TP(y(n))_{2} \,.
\]
\end{rem}

\begin{cor}\label{cor-main-tp}
There is an identification
\[ 
\lim_k \Ext_{E(\bar{\xi}_m)}^s(\mathbb{F}_2,\Gr_{\mathrm{Gre}}^*H_*(\TP(y(n))[k]))=0
\]
for $0\le m\le n$. 
\end{cor}
\begin{proof}
This follows from the proof of Proposition~\ref{yzrel} and Corollary~\ref{znmvn}, together with the fact that the unit map $\mathbb{S} \to y(n)$ makes this limit into a module over
\[ 
\lim_k \Ext_{E(\bar{\xi}_m)}^s(\mathbb{F}_2,\Gr_{\mathrm{Gre}}^*H_*(\TP(\mathbb{S})[k]))=P(v_{m})[t,t^{-1}]  \,.
\]
\end{proof}

\section{Continuous mod $2$ homology of $\TC^{-}(y(n))$}\label{Sec:TCm}
In this section, we mimic the analysis from Section~\ref{Sec:TP} in order to study the (continuous) Morava K-theory of the topological negative cyclic homology of $y(n)$. 

\subsection{Homological $\bT$-homotopy fixed point spectral sequence for $\THH(y(n))$}\label{hom fix point THH y(n)}
We now analyze the homological homotopy fixed point spectral sequence converging to the continuous homology of topological negative cyclic homology of $y(n)$. This spectral sequence has the form
\[ 
\mathrm{E}^2(n) := H^{-*}(\bT;H_*(\THH(y(n)))) \Longrightarrow H^c_*(\TC^-(y(n)))
\]
where 
\[ 
H^c_*(\TC^-(y(n)))=\underset{i}{\lim \thinspace} H_*(\TC^-(y(n))[i] )
\]
and $\TC^-(y(n))[i]:=F(E\bT^{(i)}_+,\THH(y(n)))^{\bT}$ so that $\underset{i }{\lim \thinspace}\TC^-(y(n))[i]=\TC^-(y(n))$. Note that there is also an approximate homological homotopy fixed point spectral sequence associated to the filtered object
\[ 
\mathrm{fil}_{\textup{sk}}^*\TC^{-}(y(n))[i]=\begin{cases}
\TC^{-}(y(n))[i] &\text{ if } j>i \\
\TC^{-}(y(n))[j] &\text{ if } 0\le i\le j\\
0 &\text{ if } j<0 \,. \\
\end{cases}
\]
We write 
\[
 \mathrm{Fil}_{\textup{sk}}^jH_*(\TC^{-}(y(n))[i])=H_*(\mathrm{fil}_{\textup{sk}}^j\TC^{-}(y(n)))
\]
and 
\[ 
	\mathrm{Gr}_{\textup{sk}}^jH_*(\TC^{-}(y(n))[i])=\mathrm{Fil}_{\textup{sk}}^jH_*(\TC^{-}(y(n))[i])/\mathrm{Fil}_{\textup{sk}}^{j+1}H_*(\TC^{-}(y(n))[i])\,.
\]

We will first discuss the case~$n=\infty$. This computation is entirely contained in~\cite{Rog11}, but we review it here and fill in some details for later use. 

\begin{prop}[{cf. \cite{Rog11}}]\label{Prop:Rog}
There is an equivalence
\[ 
\TC^{-}(H\mathbb{F}_2)_2\simeq \prod_{i\in \mathbb{Z}} \Sigma^{2i}H\mathbb{Z}_2 \,.
\] 
\end{prop}

\begin{proof}
The input of the homological homotopy fixed point spectral sequence is 
\[
\mathrm{E}^2(\infty) \cong P(t)\otimes \ca^{\vee}\otimes P(\sigma \bxi_1) \,,
\]
where $|t| = (-2,0)$ and an element $x$ from $\ca^\vee \otimes P(\sigma \bxi_1)$ of degree $d$ is in $(0,d)$. As in the homological Tate spectral sequence, the differentials are $t$-linear and are determined by those of the form~$d^2(x)=t\sigma x$ given by~\cite[Prop.~3.2]{BR05}. Since $\eta$ is trivial in $\THH_*(H\mathbb{F}_2)$, $\sigma$ is a derivation and therefore the only nontrivial differentials are $d^2(\bxi_i)=t\sigma \bxi_i$. Recall that $(\sigma \xi_1)^{2^i}=\sigma \bxi_{i+1}$. Then $E^3(\infty)$ is isomorphic to 
\[ P(t)\otimes P(\bxi_1^2,\bxi_2^{\prime},\dots )\oplus P(\bxi_1^2,\bxi_2^{\prime},\dots )\{(\sigma \xi_1)^k|k\ge 1\} \cong P(\bxi_1^2,\bxi_2',\ldots) \otimes [P(t) \oplus \f_2\{ (\sigma \bxi_1)^k | k \geq 1\}] \]
and the spectral sequence then collapses by~\cite[Thm.~5.1]{BR05}. 

There is an isomorphism of $\ca^{\vee}$-comodules 
\[ 
P(\bxi_1^2,\bxi_2',\ldots) \otimes [P(t) \oplus \f_2\{ (\sigma \bxi_1)^k | k \geq 1\}] \cong H_*(H\z) \otimes [P(t) \oplus \f_2\{ (\sigma \bxi_1)^k | k \geq 1\}]
\]
defined by sending $\bxi_1^{2j}$ to $\bxi_1^{2j}+(\sigma \bxi_1)^j$ for all $j \geq 0$, $\bxi_i^{\prime}$ to $\bxi_i$ for $i\ge 2$, and $t^\ell$ to $t^\ell$ for all $\ell \geq 0$, and $(\sigma \bxi_1)^k$ to $(\sigma \bxi_1)^k$ for all $k \geq 0$. 

The inverse limit Adams spectral sequence then has $\mathrm{E}_2$-page
\[ 
\Ext_{\ca^{\vee}}^{*,*}(\mathbb{F}_2, H_*(H\z) \otimes [P(t) \oplus \f_2\{ (\sigma \bxi_1)^k | k \geq 1\}])\cong \Ext_{E(\bxi_1)_*}^{*,*}(\mathbb{F}_2,P(t)\oplus \f_2\{(\sigma \xi_1)^k|k\ge 1\}) \,.
\]
Since $t$ and $\sigma \bxi_1$ are in even degrees, the $Q_0$-action is trivial and we see that this $\mathrm{E}_2$-page is isomorphic to 
\[ 
P(v_0)\otimes P(t)\oplus \f_2\{(\sigma \xi_1)^k|k\ge 1\} \,.
\]
The usual calculation for resolving extensions in this spectral sequence produces an isomorphism 
\[
\pi_*(\TC^{-}(H\f_2)_{2}) \simeq \prod_{i \in \mathbb{Z}}\pi_*(\Sigma^{2i}H\mathbb{Z}_2) \,. 
\]
The result then follows from the fact that $\TC^{-}(H\mathbb{F}_2)_2$ is a commutative $\K(\mathbb{F}_2)_2\simeq H\mathbb{Z}_2$-algebra. 
\end{proof}

\begin{rem}
In the case $n=0$, the map $\THH(\mathbb{S})\to \mathbb{S}$ induced by collapsing $\bT$ to a point is a $\bT$-equivariant equivalence and consequently the $\bT$-action on $\THH(\mathbb{S})\simeq \mathbb{S}$ is trivial. 
This implies that 
\[ 
\TC^{-}(\mathbb{S})\simeq F(\mathbb{C}P^{\infty}_+,\mathbb{S})
\]
Therefore, we expect $\TC^{-}(y(n))_{2}$ to interpolate between the $2$-completion of the Spanier--Whitehead dual of $\mathbb{C}P^{\infty}_+$ and $\prod_{i\in \mathbb{Z}} \Sigma^{2i} \mathbb{Z}_{2}$. Our computations are consistent with this expectation. 
\end{rem}

\begin{prop}\label{negbasis}
There is an isomorphism of graded $\f_2$-vector spaces 
\begin{align*}
 H^c_*(\TC^-(y(n))) \cong &  H_*(z(n)/v_n)\otimes P(t)  \oplus T 
 \end{align*}
with $\mathrm{T}$ the simple $t$-torsion module 
\[
\mathrm{T} := H_*(z(n)) \otimes \f_2\left\{  \prod_{i=1}^{n}(\sigma\bxi_i)^{\epsilon_i} : \epsilon_i \in\{0,1\}, \sum \epsilon_i\ge 1 \right\} ,
\]
where $|t| = (-2,0)$, $|\bxi_i| = (0,2^i-1)$, and $|\sigma \bxi_i| = (0,2^i)$. Moreover, there is an isomorphism of continuous $\mathcal{E}^{\vee}$-comodules 
\[
 \lim_i \Gr_{\sk}^{*}H_*(\TC^{-}(y(n))[i])\cong H_*(z(n)/v_n)\otimes P(t)  \oplus \mathrm{T} 
 \]
 where the $\mathcal{E}^{\vee}$-comodule action on $H_*(z(n))$ and $H_*(z(n)/v_n)$ was described in Corollaries \ref{mhzn} and~\ref{znmvn} and the $\mathcal{E}^{\vee}$-comodule action on $\f_2\left\{  \prod_{i=1}^{n}(\sigma\bxi_i)^{\epsilon_i} : \epsilon_i \in\{0,1\}, \sum \epsilon_i\ge 1 \right\}$ and $P(t)$ is trivial. 
\end{prop}

\begin{proof}
We will use the homological homotopy fixed point spectral sequence. Before we give the details, we note that the somewhat mysterious summand $T$ above will appear as the kernel of the $d^2$-differential of the homological homotopy fixed point spectral sequence in bidegrees~$(0,*)$, i.e. it consists of the kernel of $d^2$ along one edge of the spectral sequence.\footnote{Similar ``edge terms" appear in the computations of Bruner--Rognes~\cite[Prop.~6.1]{BR05}.} The other summand will be the homology of the $d_2$-differentials in the remaining bidegrees. 

The homological homotopy fixed point spectral sequence has $\mathrm{E}^2$-page
\[
\mathrm{E}^2(n) = H^{-*}(\bT;H_*(THH(y(n)))) \cong P(t) \otimes P(\bxi_1,\ldots,\bxi_n) \otimes E(\sigma \bxi_1,\ldots,\sigma \bxi_n)
\]
where $|t| = (-2,0)$, $|\bxi_i| = (0,2^i-1)$ and $|\sigma \bxi_i| = (0,2^i)$. As in the Tate case, $\mathrm{E}^2(n)$ is a module over $\mathrm{E}^2(0)\cong P(t)$. Therefore $d^r(t)=0$ for all $r\ge 1$ and all differentials are $t$-linear. 

The $d^2$-differentials in the homological homotopy fixed point spectral sequence are of the form~$d^2(x) = t \sigma x$ by~\cite[Prop.~3.2]{BR05}, and $\sigma$ acts as a derivation as in the Tate case. We therefore obtain an additive isomorphism
\[
\ker d^2 \cong P(t) \otimes  P(\bxi^2_1,\bxi_2',\ldots,\bxi_n') \oplus \mathrm{T} 
\]
where $\mathrm{T}:= P(\bxi^2_1,\bxi_2',\ldots,\bxi_n') \otimes T_0$ and 
\begin{equation}\label{tzero}
\mathrm{T}_0=\f_2\{ \prod_{i=1}^{n}\sigma\bxi_i^{\epsilon_i} : \epsilon_i \in\{0,1\} , \sum_i\epsilon_i\ge 1\}  \,.
\end{equation}
The summand $\mathrm{T}$ can be identified with a summand in $H_*(\THH(y(n)))$, up to intermediary filtration, by sending $a \otimes b \in \mathrm{T}$ to $ab \in H_*(\THH(H\f_2))$; the product $ab$ has a unique lift up to intermediary filtration in $H_*(\THH(y(n)))$ by the computations above. 

We therefore just need to compute $\im d^2\subset \ker d^2$ to identify $\mathrm{E}^3(n)$. First, note that $t\sigma \bxi_i$ is in $\im d^2$ for all $1\le i\le n$ since $d^2(\bxi_i)=t\sigma \bxi_i$. Also, no element of the form~$x\in P(t)\otimes P(\bxi^2_1,\ldots,\bxi^2_n)\otimes P(\bxi_2',...,\bxi_n')$ is in $\im d^2$. Finally, observe that $t\bxi_ix_i +t\bxi_jx_j\in \im d^2$ for $1\le i < j \le n$ as in the proof of Proposition \ref{tatebasis}. ~We conclude that 
\[ 
\im d^2 = \left[ P(t) \otimes E(\sigma \bxi_1, \ldots, \sigma \bxi_n) \otimes \f_2\{y \cdot \bxi_ix_i + y \bxi_jx_j : 1\le i < j \le  n \text{ and } y \in \ker d^2\} \right] \{t\} \,.
\]

Thus, up to a change of basis, the class $t\bxi_nx_n$ survives to the $\mathrm{E}^3(n)$-page. We can therefore identify the $\mathrm{E}^3$-page as
\begin{align*} 
\mathrm{E}^3_{**}(n) \cong P(\bxi^2_1,\bxi_2',\ldots \bxi_n' )\otimes E(\bxi_{n}x_n)\otimes  P(t) \oplus \mathrm{T} 
\end{align*}
 with $\mathrm{T}:= P(\bxi^2_1,\bxi_2',\ldots,\bxi_n') \otimes \mathrm{T}_0$  and $\mathrm{T}_0$ defined in~\eqref{tzero}.
 
To see that there are no further differentials, we use the map of homological $\bT$-homotopy fixed point spectral sequences induced by the $\bT$-equivariant map $\THH(y(n)) \to \THH(H\f_2)$. 
The homological homotopy fixed point spectral sequence converging to the graded $\f_{2}$-vector spaces $H^c_*(\TC^{-}(H\f_2))$ has $\mathrm{E}^3$-page
\[
\mathrm{E}^3(\infty) \cong P(t) \otimes P(\bxi_1^2, \bxi_{i+1}' : i \geq 1)  \oplus  P(\bxi_1^2,\bxi_{i+1}' : i \geq 1)\otimes \f_2\{ (\sigma \bxi_1)^k : k\geq 1 \} \,.
\]
By Proposition~\ref{Prop:ymapF} and the fact that all $d^2$-differentials in the source also occur in the target, the map is injective on $\mathrm{E}^3$-pages. Since there is an isomorphism~$\mathrm{E}^3(\infty)\cong \mathrm{E}^{\infty}(\infty)$ by~\cite[Thm.~5.1]{BR05}, there are isomorphisms~$\mathrm{E}^3(n)\cong \mathrm{E}^\infty(n)$ for all $n>0$. 

As in the proof of Proposition~\ref{yzrel}, we can identify 
\[ 
	P(\bxi^2_1, \bxi_2',\ldots \bxi_n')\otimes E(\bxi_{n}\sigma \bxi_n)\otimes P(t)=H_*(z(n)/v_n)\otimes P(t)
\] 
and 
\[ 
	\mathrm{T}\cong  H_*(z(n))\otimes \f_2\{\prod_{i=1}^n (\sigma \bar{\xi}_i)^{\epsilon_i}: \epsilon_i \in \{0,1\}, \sum \epsilon_i \geq 1\}
\]
as continuous $\mathcal{E}^\vee$-comodules and consequently, we also identify 
\[
	\mathrm{Gr}_{\textup{Gre}}^*H_*(\TC^{-}(y(n)))\cong H_*(z(n)/v_n)\otimes P(t)\oplus H_*(z(n))\otimes \f_2\{\prod_{i=1}^n (\sigma \bar{\xi}_i)^{\epsilon_i}: \epsilon_i \in \{0,1\}, \sum \epsilon_i \geq 1\}
\] 
as $\mathcal{E}^\vee$-comodules. 
\end{proof}

\begin{cor}\label{cor-main-tn}
There is an identification
\[ 
	\lim_k \Ext_{E(\bxi_{m})}^{s,*}(\mathbb{F}_2,\Gr_{\textup{Gre}}^{*}H_{*}(\TC^{-}(y(n))[k];\mathbb{F}_{2}))=0
\]
for $s>0$. 
\end{cor}
\begin{proof}
The analogue of the last display equation in the proof of  Proposition~\ref{negbasis} implies that 
$$\Gr_{\textup{Gre}}^{*} H_*(\TC^-(y(n))[k];\f_2)$$
can be expressed in terms of $H_*(z(n)/v_n)$ and $H_*(z(n))$ (tensored with appropriate comodules). The relevant Margolis homology groups then vanish for all $k$ in view of Corollaries \ref{mhzn} and~\ref{znmvn}, so the limit under consideration is a limit of trivial groups and thus vanishes. 
\end{proof}

\section{The inverse limit May--Ravenel spectral sequence}\label{Sec:MR}

In this section, we pass from continuous homology to continuous Morava K-theory using inverse limit Adams spectral sequences~\cite{LNR11}. The main technical difficulty of this approach is computing the $\mathrm{E}_2$-pages of these spectral sequences, which boils down to resolving hidden $\mathcal{E}^{\vee}$-comodule extensions in the homological homotopy fixed point and Tate spectral sequences. We do this using a new technical tool, the ``inverse limit May--Ravenel spectral sequence,''\footnote{The spectral sequence from \cite{Rav86} generalizing~\cite{May65} is called the Ravenel--May spectral sequence in~\cite{Sal23} and Ravenel spectral sequence in~\cite{HW22}. We prefer to give credit to both authors and use alphabetical order.} which we develop in Section~\ref{SS:MR}. We apply the spectral sequence to complete the proof of the main theorem in Section~\ref{SS:Finish}. 

\subsection{The inverse limit May--Ravenel spectral sequence}\label{SS:MR}

In foundational work of May~\cite{May65},  he constructs a spectral sequence associated to a filtered Hopf algebra. This is generalized in work of Ravenel~\cite{Rav86}. Here we give a further generalization that may be of independent interest. 

Throughout this section, let $(\mathbb{F}_{p},\Gamma)$ be a graded commutative, connective, flat Hopf algebroid such that $\Gamma$ is a finite type graded $\mathbb{F}_{p}$-module. We write $\overline{\Gamma}=\mathrm{coker}(\eta_{L})$ for the cokernel of the left unit. In the following construction, we will work in the abelian category $\mathrm{Ch}(\mathbb{F}_{p},\Gamma)$ of chain complexes of graded $\Gamma$-comodules, where we write $\mathrm{CoMod}(\mathbb{F}_{p},\Gamma)$ for the abelian category of graded $\Gamma$-comodules and simply $\mathrm{Cotor}_{(\mathbb{F}_{p},\Gamma)}^{s}(X,-)$ for the derived functors of the cotensor $X\square-: \mathrm{CoMod}(\mathbb{F}_{p},\Gamma)\to \mathrm{Ab}$ where $\mathrm{Ab}$ denotes the category of abelian groups. We also abbreviate and write $\mathrm{Pro}(\mathbb{F}_{p},\Gamma):=\mathrm{Pro}(\mathrm{Comod}(\mathbb{F}_{p},\Gamma))$ and $\mathrm{Ind}(\mathbb{F}_{p},\Gamma):=\mathrm{Ind}(\mathrm{Comod}(\mathbb{F}_{p},\Gamma))$. We write $\mathrm{Pro}(\mathbb{F}_{p},\Gamma)^{\textup{f.t.}}$ for the full subcategory of $\mathrm{Pro}(\mathbb{F}_{p},\Gamma)$ spanned by those pro-objects in $(\mathbb{F}_{p},\Gamma)$ that are object-wise finite type. We note that this forms an abelian category. 

\begin{const}\label{mayravss}
Let $\{ M_{i}\}$ be in $\mathrm{Pro}(\mathbb{F}_{p},\Gamma)^{\textup{f.t.}}$. 
Let $\{ \mathrm{Fil}^{\star} M_{i}\}$  be a decreasingly filtered object in $\mathrm{Pro}(\mathbb{F}_{p},\Gamma)^{\textup{f.t.}}$ such that all the structure maps $\{ \mathrm{Fil}^{s} M_{i}\}\to \{ \mathrm{Fil}^{s-1} M_{i}\}$ are level maps of pro-objects, 
$\lim_{s}\{ \mathrm{Fil}^{s}M_{i}\}=\{M_{i}\}\,$,  and $\colim_{s}\{\mathrm{Fil}^{s}M_{i}\} = 0$. For each fixed $i$, we write $D_{\Gamma}^{\bullet}(M_{i})$ and $C_{\Gamma}^{\bullet}(M_{i})$, following the convention in \cite[Definition~A1.2.11]{Rav86}, and we further 
consider 
\[\mathrm{Fil}^{\star}\{ D^\bullet_{\Gamma}(M_{i})\} : =\{ \Gamma \otimes_{\mathbb{F}_{p}}\overline{\Gamma}^{\otimes_{\mathbb{F}_{p}}\bullet} \otimes_{\mathbb{F}_{p}}\mathrm{Fil}^{\star}M_{i} \} \]
and 
\begin{equation}\label{filtered object for inverse limit may ravenel}
\mathrm{Fil}^{\star}\{ C_{\Gamma}^{\bullet}(M_{i})\} : =\{ \mathbb{F}_{p}\square_{\Gamma} \big ( \Gamma \otimes_{\mathbb{F}_{p}}\overline{\Gamma}^{\otimes_{\mathbb{F}_{p}}\bullet} \otimes_{\mathbb{F}_{p}}\mathrm{Fil}^{\star}M_{i}\big ) \} 
\end{equation}
as filtered objects in $\mathrm{Pro}(\mathrm{Ch}(\mathbb{F}_{p},\Gamma))$. We also write 
\[\mathrm{Gr}^{s}\{ D^\bullet_{\Gamma}(M_{i})\} : =\mathrm{Fil}^{s}\{ D^\bullet_{\Gamma}(M_{i})\}/\mathrm{Fil}^{s+1}\{ D^\bullet_{\Gamma}(M_{i})\}  \]
and 
\begin{equation}\label{filtered object for inverse limit may ravenel}
\mathrm{Gr}^{s}\{ C_{\Gamma}^{\bullet}(M_{i})\} : =\mathrm{Fil}^{s}\{ C_{\Gamma}^{\bullet}(M_{i})\}/\mathrm{Fil}^{s+1}\{ C_{\Gamma}^{\bullet}(M_{i}) \} \,.
\end{equation}
We call the spectral sequence produced by applying homology to the filtered chain complex of abelian groups
\[\lim_{i}\mathrm{Fil}^{\star}\{ C_{\Gamma}^{\bullet}(M_{i})\}\]  
the \emph{inverse limit May--Ravenel spectral sequence}. 
\end{const}
 
\begin{lem}\label{lem:cotor-pro}
Suppose that $\{N_{i}\}$ is in $\mathrm{Pro}(\mathbb{F}_{p},\Gamma)^{\textup{f.t.}}$. 
Then there is an isomorphism
\begin{equation}\label{eq:cotor-pro}
\Cotor_{\mathrm{Pro}(\mathbb{F}_{p},\Gamma)}^{*}(\mathbb{F}_{p},\{N_{i}\})\cong  \lim_{i} \Cotor_{(\mathbb{F}_{p},\Gamma)}^{*}(\mathbb{F}_{p},N_{i}) \,.
\end{equation}
\end{lem}

\begin{proof}
Note that $\Gamma \otimes_{\mathbb{F}_{p}}\overline{\Gamma}^{\otimes_{\mathbb{F}_{p}}t} \otimes_{\mathbb{F}_{p}}N_{i}$ is a relative injective $\Gamma$-comodule for each $t$, and $i$ and the levelwise cotensor $\mathbb{F}_{p}\square-$ is exact on levelwise relative injective pro-objects in $\Gamma$-comodules. The composite $\lim (\mathbb{F}_{p}\square -)$ is exact on objects in $\mathrm{Pro}(\mathbb{F}_{p},\Gamma)^{\textup{f.t.}}$ that are object-wise relative injective. 
The isomorphism then holds by using the fact that  $\mathrm{Pro}(\mathrm{Ch}(\mathbb{F}_{p},\Gamma))=\mathrm{Ind}(\mathrm{Ch}(\mathbb{F}_{p},\Gamma)^{\textup{op}})^{\textup{op}}$ and the analogue of \cite[Corollary 15.3.9]{KS06} for pro-objects where we replace $\mathrm{Hom}_{\mathrm{Ind}(\mathcal{C})}(X,\--):\mathrm{Ind}(\mathcal{C}) \to \mathrm{Ab}$ with the composite functor 
\[\lim (\mathbb{F}_{p}\square \--):\mathrm{Pro}(\mathbb{F}_{p},\Gamma)^{\textup{f.t.}}\to \mathrm{Ab} \,.\] 

Now, observe that the functor $\mathbb{F}_p \square -$ sends relative injective finite type $(\mathbb{F}_p,\Gamma)$-comodules to $\lim$-acyclic $(\mathbb{F}_p,\Gamma)$-comodules since all levelwise finite type pro-objects in $(\mathbb{F}_p,\Gamma)$-comodules are $\lim$-acyclic (cf.~\cite[Lemma~15.16]{Was82}). Since $\mathrm{Pro}(\mathbb{F}_{p},\Gamma)^{\textup{f.t.}}$ is an abelian category, we can therefore consider the Grothendieck spectral sequence for the composite 
\[ \lim  (\mathbb{F}_{p}\square \--): \mathrm{Pro}(\mathbb{F}_{p},\Gamma)^{\textup{f.t.}}\to \mathrm{Ab} \,.\] 
This Grothendieck spectral sequence collapses to the line given by $\lim \Cotor_{(\mathbb{F}_{p},\Gamma)}^{s}(\mathbb{F}_{p},-)$, again since all finite type $(\mathbb{F}_p,\Gamma)$-comodules are $\lim$-acyclic, yielding the desired isomorphism. 
\end{proof}

\begin{prop}\label{prop:E2page}
We can identify the $\mathrm{E}_{2}$-page of the spectral sequence from Construction~\ref{mayravss} with 
\[ \lim_{i}\Cotor_{(\mathbb{F}_{p},\Gamma)}^{*}(\mathbb{F}_{p},\mathrm{Gr}^{\ast}M_{i}) \] 
whenever $M_{i}$ is a finite type $\mathbb{F}_{p}$-module for each $i$.  
\end{prop}
\begin{proof}
Since kernels and cokernels (and consequently images) of level maps are computed levelwise, we observe that the $\mathrm{E}_{2}$-page from Construction~\ref{mayravss} satisfies
\[ 
\mathrm{E}_{2}^{*,*} = H_{*}( \mathrm{lim}_{i} \mathrm{Gr}^{\ast}\{ C_{\Gamma}^{\bullet}(M_{i})\} ) \,.
\]
We then determine that 
\begin{align}
\label{l1}H_{*}( \mathrm{lim}_{i} \mathrm{Gr}^{\ast}\{ C_{\Gamma}^{\bullet}(M_{i})\} ) &  \cong \Cotor_{\mathrm{Pro}(\mathbb{F}_{p},\Gamma)}^{*}(\mathbb{F}_{p},\{\mathrm{Gr}^{\ast}M_{i}\}) \\
\label{l2} & \cong \lim_{i} \Cotor_{(\mathbb{F}_{p},\Gamma)}^{*}(\mathbb{F}_{p},\mathrm{Gr}^{\ast}M_{i})   
\end{align}
where we write $\Cotor_{\mathrm{Pro}(\mathbb{F}_{p},\Gamma)}^{*}(\mathbb{F}_{p},\{\mathrm{Gr}^{\ast}M_{i}\})$ for the derived functors of the functor 
\[ 
\mathrm{lim} (\mathbb{F}_{p }\square \--): \mathrm{Pro}(\mathbb{F}_{p},\Gamma)^{\textup{f.t.}}\to \mathrm{Ab} \,.
\] 
The isomorphism~\eqref{l1} then holds because $\Gamma \otimes_{\mathbb{F}_{p}}\overline{\Gamma}^{\otimes_{\mathbb{F}_{p}}\bullet} \otimes_{\mathbb{F}_{p}}\mathrm{Fil}^{\star}M_{i}$ and $\Gamma \otimes_{\mathbb{F}_{p}}\overline{\Gamma}^{\otimes_{\mathbb{F}_{p}}\bullet} \otimes_{\mathbb{F}_{p}}\mathrm{Gr}^{\ast}M_{i}$ are relative injective $\Gamma$-comodules for each $\star$, $\bullet$, and $k$ and the levelwise cotensor~$\mathbb{F}_{p}\square\--$ is exact on levelwise relative injective pro-obects in $\Gamma$-comodules. Moreover, the composite $\lim (\mathbb{F}_{p}\square \--)$ is exact on objects in $\mathrm{Pro}(\mathbb{F}_{p},\Gamma)^{\textup{f.t.}}$ that are object-wise relative injective $\Gamma$-comodules, such as $\Gamma \otimes_{\mathbb{F}_{p}}\overline{\Gamma}^{\otimes_{\mathbb{F}_{p}}s} \otimes_{\mathbb{F}_{p}}\mathrm{Gr}^{\ast}M_{i}$. The isomorphism~\eqref{l2} then holds by Lemma~\ref{lem:cotor-pro}. 
\end{proof}

\begin{prop}\label{prop:conditional-convergence-May--Ravenel}
Suppose $\{M_i\}$ is in $\mathrm{Pro}(\mathbb{F}_p,\Gamma)^{\textup{f.t.}}$ and the filtered objects~$\mathrm{Fil}^fM_{i}$ in $\mathrm{Pro}(\mathbb{F}_p,\Gamma)^{\textup{f.t.}}$ from Construction~\ref{mayravss} are eventually constant for each $i$ and satisfy
\[ 
 \lim_{f} \left ( \mathrm{Fil}^f\{ C_{\Gamma}^{\bullet}(M_{i})\}  \right )=C_{\Gamma}^{\bullet}(M_{i})
 \]
for each $i$. Then the inverse limit May--Ravenel spectral sequence conditionally converges to 
\[ 
\lim_{i}\mathrm{Cotor}_{(\mathbb{F}_p,\Gamma)}^{s}(\mathbb{F}_p,M_{i})\,.
\]
\end{prop}
\begin{proof}
Since 
\[ 
\colim_{f} \{ \mathrm{Fil}^fM_{i}\}=0 \,,
\]
we know that 
\[ 
\colim_{f} \mathrm{Fil}^f\{ C_{\Gamma}^{\bullet}(M_{i})\}=0
\]
so the result follows because passing to homology commutes with filtered colimits. Moreover, we note that  
\begin{align}
\label{m1} H_{*} \left(\lim_{f}  ( \mathrm{Fil}^f\{ C_{\Gamma}^{\bullet}(M_{i})\} ) \right ) & \cong  \Cotor_{\mathrm{Pro}(\mathbb{F}_{p},\Gamma)}^{*}(\mathbb{F}_{p}, \{ M_{i}\})\\
\label{m2}  & =\lim_{i}\Cotor_{(\mathbb{F}_{p},\Gamma)}^{*}(\mathbb{F}_{p},M_{i}) \,.
\end{align}
Here the isomorphism \eqref{m1} holds because $\mathrm{Fil}^fM_{i}$ is eventually constant for each $i$ and the isomorphism~\eqref{m2} holds by Lemma~\ref{lem:cotor-pro}. 
Therefore, the inverse limit May--Ravenel spectral sequence conditionally converges to the limit in the sense of~\cite[Definition 5.10]{Boa99} under the stated hypotheses.
\end{proof} 

\begin{prop}\label{prop:inverse limit-May--Ravenel}
Set $\Gamma=E(\bar{\xi}_{m})$ and suppose $\mathrm{Fil}^{s}M_{i}$ from Construction~\ref{mayravss} is eventually constant as $s \to  -\infty$ for each $i$. Suppose the inverse limit May--Ravenel spectral sequence has a horizontal vanishing line of slope zero and postive $y$-intercept. Further, assume that $M_{i}$ is a finite type~$\Gamma$-comodule for each~$i$. Then the inverse limit May--Ravenel spectral sequence has $\mathrm{E}_{2}$-page 
 \[ 
 \mathrm{E}_{2}^{s,t,*} = \lim_{i}\Cotor_{(\mathbb{F}_p,E(\bar{\xi}_{m}))}^{s}(\mathbb{F}_{p},\mathrm{Gr}^{t} M_{i}) \,,
 \]
differential 
\[
d_{r}^{s,t,u} \colon \thinspace  \mathrm{E}_{2}^{s,t,u} \to \mathrm{E}_{2}^{s+1,t-r,u} \,,
\]
and strongly converges to 
\[ 
\lim_{i}\Cotor_{(\mathbb{F}_p,E(\bar{\xi}_{m}))}^{s,\ast}(\mathbb{F}_{p},M_{i}) 
\]
whenever $M_{i}$ is a bounded below for each~$i$.
\end{prop}
\begin{proof}
Conditional convergence to 
\[ 
\lim_{i}\Cotor_{(\mathbb{F}_p,E(\bar{\xi}_{m}))}^{s,\ast}(\mathbb{F}_{p},M_{i}) 
\]
follows from Proposition~\ref{prop:conditional-convergence-May--Ravenel} since 
\[
 \lim_{f} \left ( \mathrm{Fil}^f\{ C_{\Gamma}^{\bullet}(M_{i})\}  \right )=C_{\Gamma}^{\bullet}(M_{i}) \,
 \]
 holds whenever $ \mathrm{Fil}^fM_{i}$ is eventually constant as $f \to -\infty$ for each $i$ and we know that $C_{\Gamma}^{s}(\mathrm{Gr}^{t}M_{i})$ and $C_{\Gamma}^{s}(M_{i})$ are in $\mathrm{Pro}(\mathbb{F}_p,E(\bar{\xi}_m))^{\textup{f.t.}}$ for each $s$ and $t$ under our finite type and bounded below hypotheses. The horizontal vanishing line and the finite type hypothesis also imply that the obstruction~$\mathrm{W}$ from~\cite[Lem.~8.5]{Boa99} vanishes by~\cite[Lem.~8.1]{Boa99}. The vanishing line also implies that $Z_{r}^{s}$ from~\cite[p.~63]{Boa99} is eventually constant and therefore $R\mathrm{E}_{\infty}=0$ by definition (cf.~\cite[Sec.~5, Eq.~(51)]{Boa99}). Consequently, strong convergence follows from~\cite[Theorem~8.2]{Boa99}. 
\end{proof} 

\begin{rem}
Note that by~\cite[Corollary~A1.2.12]{Rav86}, we can identify 
\[ \Cotor_{(\mathbb{F}_{p},E(\xi_{m}))}^{*}(\mathbb{F}_{p},M)=\Ext_{E(\xi_{m})}^{*}(\mathbb{F}_{p},M) \,. \]
\end{rem}

\subsection{Final computation}\label{SS:Finish}

We now specialize back to homological trace methods for $y(n)$. As a consequence of the results in the previous subsection, we have the following. 

\begin{prop}\label{prop--inverse limit-yn}
There are strongly convergent inverse limit May--Ravenel spectral sequences 
\[ 
\lim_{i} \Ext_{E(\xi_{m})}^{s,\ast}(\mathbb{F}_{2},\mathrm{Gr}^{\ast}_{\textup{sk}}H_{*}(\TC^{-}(y(n))[i]))\implies  \lim_{i}\Ext_{E(\xi_{m})}^{s,\ast}(\mathbb{F}_{2},H_{*}(\TC^{-}(y(n))[i]))
\]
and 
\[ 
\lim_{i} \Ext_{E(\xi_{m})}^{s,\ast}(\mathbb{F}_{2},\mathrm{Gr}^{\ast}_{\textup{Gre}}H_{*}(\TP(y(n))[i])\implies \lim_{i}\Ext_{E(\xi_{m})}^{s,\ast}(\mathbb{F}_{2},H_{*}(\TP(y(n))[i])\}
\]
where $\mathrm{Gr}^{\ast}_{\textup{sk}}H_{*}(\TC^{-}(y(n))[i])$ denotes the $\mathrm{E}_{\infty}$-page of the approximate homological homotopy fixed point spectral sequence and $\mathrm{Gr}^{\ast}_{\textup{Gre}}H_{*}(\TC^{-}(y(n))[i])$ denotes the $\mathrm{E}_{\infty}$-page of the approximate homological Tate spectral sequence. 
\end{prop}
\begin{proof}
By Proposition~\ref{prop:inverse limit-May--Ravenel}, it suffices to check that $H_{*}(\TC^{-}(y(n))[i])$ and $H_{*}(\TP(y(n))[i])$ are finite type, bounded below, and the filtrations 
\[ 
\mathrm{Fil}^{\ast}_{\textup{sk}}H_{*}(\TC^{-}(y(n))[i]) \quad \text{ and } \quad \mathrm{Fil}^{\ast}_{\textup{Gre}}H_{*}(\TP(y(n))[i])
\] 
are eventually constant, all of which are clear from the definitions and the computations in the proofs of Propositions \ref{yzrel} and~\ref{negbasis}. 
\end{proof}

\begin{rem}\label{hidden-extensions}
The existence of the spectral sequence from Proposition~\ref{prop--inverse limit-yn} allows us to resolve hidden $\mathcal{E}^{\vee}$-module extensions in the homological Tate spectral sequence and the homological homotopy fixed point spectral sequence, as we summarize in the following result. See \cite{LN05} and~\cite{AKS20} for related results. 
\end{rem}

\begin{thm}\label{thm:continuous-Morava-k-theory-vanishing}
The $k(m)_{*}$-module $k(m)_{*}^{c}(\TP(y(n)))$ is simple $v_{m}$-torsion for all $0<m\le n$. Moreover, the $k(m)_{*}$-module $k(m)_{*}^{c}(\TC^-(y(n)))$ is simple $v_{m}$-torsion for each $0<m<n$. 
\end{thm}

\begin{proof}
It suffices to show that the $\mathrm{E}_{2}$-page of the inverse limit Adams spectral sequence 
\[ 
\lim_{i}\Ext_{\mathcal{A}^{\vee}}^{s,*}(\mathbb{F}_{2},H_{*}(k(m)\otimes \TP(y(n))[i];\mathbb{F}_{2}))\implies k(m)_{*}^{c}\TP(y(n))
\]
vanishes for $s>0$ and $0<m\le n$ and 
\[ 
\lim_{i}\Ext_{\mathcal{A}^{\vee}}^{s,*}(\mathbb{F}_{2},H_{*}(k(m)\otimes \TC^{-}(y(n))[i];\mathbb{F}_{2}))\implies k(m)_{*}^{c}\TC^{-}(y(n))
\]
vanishes for $s>0$ and $0<m<n$. By change-of-rings, it suffices to show that there is an isomorphism 
\[
\lim_i \Ext_{E(\bxi_{m})}^{s,*}(\mathbb{F}_{2},H_{*}(\TP(y(n))[i];\mathbb{F}_{2})) = 0
\]
for $s>0$ and $0<m\le n$ and there is an isomorphism 
\[ 
\lim_i \Ext_{E(\bxi_{m})}^{s,*}(\mathbb{F}_{2},H_{*}( \TC^{-}(y(n))[i];\mathbb{F}_{2})) = 0
\]
for $s>0$ and $0<m<n$. By Proposition~\ref{prop--inverse limit-yn}, the relevant inverse limit May--Ravenel spectral sequences strongly converge, so it suffices to observe that by Corollary~\ref{cor-main-tp}, the input 
\[
 \lim_i \Ext_{E(\bxi_{m})}^{s,*}(\mathbb{F}_{2},\Gr_{\textup{Gre}}^{*}H_{*}(\TP(y(n))[i];\mathbb{F}_{2})) 
 \]
is zero for $s>0$ for all $0<m \le n$ and by Corollary~\ref{cor-main-tn}, the input 
\[ 
\lim_i \Ext_{E(\bxi_{m})}^{s,*}(\mathbb{F}_{2},\Gr_{\textup{Gre}}^{*}H_{*}(\TC^{-}(y(n))[i];\mathbb{F}_{2})) 
\]
is zero for $s>0$ for all $0<m<n$. 
\end{proof}

\begin{rem}\label{rem:continuos-Morava-K-theory-to-Morava-K-theory}
Let $F\in \{\TC^{-},\TP\}$. By Adams' theorem~\cite[Theorem III.15.2]{Ada95}, there is an equivalence 
\[\lim_{i}k(m)\wedge \tau_{\ge w}F(y(n))[i]\simeq k(m)\wedge \lim_{i}\tau_{\ge w}F(y(n))[i] \]
for each integer $w$. Note that since the poset $(\mathbb{Z},\le )$ is a $1$-category, we can write this homotopy limit as a fiber between infinite products (using the Bousfield--Kan formula for the homotopy limit) and it suffices to commute $k(n)$ with these infinite products, which is the content of~\cite[Theorem III.15.2]{Ada95}. 
Since a sequential limit of uniformly bounded above spectra is bounded above and bounded above spectra are $K(m)$-acyclic, we know that 
\[ K(m)_{*}\lim_{i} \tau_{\le w-1} F(y(n))[i] =0 \]
for each integer $w$. Therefore, there are equivalences 
\begin{align*}
v_{m}^{-1}\lim_{i}k(m)\wedge \tau_{\ge w}F(y(n))[i] & \simeq K(m) \wedge \lim_{i} \tau_{\ge w}F(y(n))[i] \\ 
& \simeq K(m)\wedge  \lim_{i} F(y(n))[i]  \\ 
& \simeq K(m)\wedge  F(y(n))  
\end{align*}
for each integer $w$. Consequently, there are long exact sequences 
\[ 
\dots \to v_{m}^{-1} k(m)_{s}^{c}F(y(n))\to v_{m}^{-1}k(m)_{s}^{c}(\tau_{\le w}F(y(n)))\to K(m)_{s+1}F(y(n)) \to \dots 
\]
relating $v_{m}^{-1} k(m)_{s}^{c}F(y(n))$ to $K(m)_{s+1}F(y(n))$ for $F\in \{\TC^{-},\TP\}$. Here we write 
\[v_{m}^{-1} k(m)_{s}^{c}F(y(n)) := v_{m}^{-1}\lim_{i}k(m)_{s}F(y(n))[i] \]
and 
\[v_{m}^{-1}k(m)_{s}^{c}(\tau_{\le w}F(y(n))): =v_{m}^{-1}\lim_{i}k(m)_{s}(\tau_{\le w}F(y(n))[i]) \,.\]

Note that by~\cite{NS18}, there is a long exact sequence 
\[
\cdots  \to K(m)_{*}\TC(y(n))\to K(m)_{*}\TC^{-}(y(n))\to K(m)_{*}\TP(y(n))\to \cdots 
\]
for all $m\ge 1$ and by~\cite{DGM12,HM97} the trace map 
\[ 
K(m)_{*}\K(y(n))\to K(m)_{*}\TC(y(n))
\]
is an isomorphism for $m\ge 1$. So vanishing of $K(m)_{*}\TC^{-}(y(n))$ and $K(m)_{*}\TP(y(n))$ for $m\ge 1$ implies vanishing of $K(m)_{*}\K(y(n))$ for $m\ge 1$. 

Since the first version of this paper appeared in preprint form, work of~\cite{LMMT24} independently showed that 
\[
K(m)_{*}\K(y(n))=0
\] 
for $0<m<n$. This suggests that
\[ 
v_{m}^{-1}k(m)_{s}^{c}(\tau_{\le w}\TP(y(n)))=v_{m}^{-1}k(m)_{s}^{c}(\tau_{\le w}\mathrm{TC}^{-}(y(n)))=0
\]
for $0<m<n$. In joint work with Salch~\cite{AKS20}, the first author worked towards proving this, but unfortunately in the course of revising op.~cit.~an additional hypothesis was deemed necessary, which we have not been able to verify in the case of $y(n)$. 
\end{rem}

\newcommand{\etalchar}[1]{$^{#1}$}


\begin{thebibliography}{GJMS{\etalchar{+}}06}

\bibitem[ABG{\etalchar{+}}14]{ABGHR14}
Matthew Ando, Andrew~J. Blumberg, David Gepner, Michael~J. Hopkins, and Charles
  Rezk.
\newblock Units of ring spectra, orientations and {T}hom spectra via rigid
  infinite loop space theory.
\newblock {\em J. Topol.}, 7(4):1077--1117, 2014.

\bibitem[ABG18]{ABG18}
Matthew Ando, Andrew~J. Blumberg, and David Gepner.
\newblock Parametrized spectra, multiplicative {T}hom spectra and the twisted
  {U}mkehr map.
\newblock {\em Geom. Topol.}, 22(7):3761--3825, 2018.

\bibitem[Ada95]{Ada95}
John~Frank Adams.
\newblock {\em Stable homotopy and generalised homology}.
\newblock University of Chicago press, 1995.

\bibitem[AKHW24]{AKHW24}
Gabriel Angelini-Knoll, Jeremy Hahn, and Dylan Wilson.
\newblock Syntomic cohomology of {Morava} {K}-theory.
\newblock Preprint, {arXiv}:2410.07048 [math.{KT}] (2024), 2024.

\bibitem[Ang08]{Ang08}
Vigleik Angeltveit.
\newblock Topological {H}ochschild homology and cohomology of ${A}_\infty$ ring
  spectra.
\newblock {\em Geometry \& Topology}, 12(2):987--1032, 2008.

\bibitem[AR05]{AR05}
Vigleik Angeltveit and John Rognes.
\newblock Hopf algebra structure on topological {H}ochschild homology.
\newblock {\em Algebraic \& Geometric Topology}, 5(3):1223--1290, 2005.

\bibitem[AR08]{AR08}
Christian Ausoni and John Rognes.
\newblock The chromatic red-shift in algebraic {K}-theory.
\newblock {\em Enseignement Math{\'e}matique}, 54(2):13--15, 2008.

\bibitem[AS20]{AKS20}
Gabriel {Angelini-Knoll} and Andrew {Salch}.
\newblock {Commuting unbounded homotopy limits with Morava K-theory}.
\newblock {\em arXiv e-prints}, page arXiv:2003.03510, March 2020.

\bibitem[BCS10]{BCS10}
Andrew~J Blumberg, Ralph~L Cohen, and Christian Schlichtkrull.
\newblock Topological {H}ochschild homology of {T}hom spectra and the free loop
  space.
\newblock {\em Geometry \& Topology}, 14(2):1165--1242, 2010.

\bibitem[BL22]{BL22}
Bhargav {Bhatt} and Jacob {Lurie}.
\newblock {Absolute prismatic cohomology}.
\newblock {\em arXiv e-prints}, page arXiv:2201.06120, January 2022.

\bibitem[BMMS06]{BMMS06}
Robert~R Bruner, J~Peter May, James~E McClure, and Mark Steinberger.
\newblock {\em Hoo ring spectra and their applications}, volume 1176.
\newblock Springer, 2006.

\bibitem[BMS19]{BMS18}
Bhargav Bhatt, Matthew Morrow, and Peter Scholze.
\newblock Toological {H}ochschild homology and integral $p$-adic {H}odge
  theory.
\newblock {\em Publications Math{\'e}matiques de l'IH{\'E}S}, 129(1):199--310,
  2019.

\bibitem[Boa99]{Boa99}
J.~Michael Boardman.
\newblock Conditionally convergent spectral sequences.
\newblock In {\em Homotopy invariant algebraic structures ({B}altimore, {MD},
  1998)}, volume 239 of {\em Contemp. Math.}, pages 49--84. Amer. Math. Soc.,
  Providence, RI, 1999.

\bibitem[B{\"o}k85]{BokZ}
Marcel B{\"o}kstedt.
\newblock Topological {H}ochschild homology of $\mathbb{Z}$ and $\mathbb{Z}/p$.
\newblock {\em Preprint, Universit{\'e}it Bielefeld}, 1985.

\bibitem[BP25]{BP25}
Robert Burklund and Piotr Pstr{\k{a}}gowski.
\newblock Quivers and the {Adams} spectral sequence.
\newblock {\em Adv. Math.}, 471:76, 2025.
\newblock Id/No 110270.

\bibitem[BR05]{BR05}
Robert~R Bruner and John Rognes.
\newblock Differentials in the homological homotopy fixed point spectral
  sequence.
\newblock {\em Algebraic \& Geometric Topology}, 5(2):653--690, 2005.

\bibitem[BSY22]{BSY22}
Robert {Burklund}, Tomer~M. {Schlank}, and Allen {Yuan}.
\newblock {The Chromatic Nullstellensatz}.
\newblock {\em arXiv e-prints}, page arXiv:2207.09929, July 2022.

\bibitem[BW03]{BW03}
Tomasz Brzezinski and Robert Wisbauer.
\newblock {\em Corings and comodules}, volume 309 of {\em London Mathematical
  Society Lecture Note Series}.
\newblock Cambridge University Press, Cambridge, 2003.

\bibitem[Cha80]{Cha80}
Ruth~M. Charney.
\newblock Homology stability for {${\rm GL}\sb{n}$} of a {D}edekind domain.
\newblock {\em Invent. Math.}, 56(1):1--17, 1980.

\bibitem[CMT81]{CMT81}
F.~R. Cohen, J.~P. May, and L.~R. Taylor.
\newblock {$K({\bf Z},\,0)$} and {$K(Z_{2},\,0)$} as {T}hom spectra.
\newblock {\em Illinois J. Math.}, 25(1):99--106, 1981.

\bibitem[Dev24]{Dev24}
Sanath~K. Devalapurkar.
\newblock Higher chromatic {Thom} spectra via unstable homotopy theory.
\newblock {\em Algebr. Geom. Topol.}, 24(1):49--108, 2024.

\bibitem[DGM12]{DGM12}
Bj{\o}rn~Ian Dundas, Thomas~G Goodwillie, and Randy McCarthy.
\newblock {\em The local structure of algebraic {K}-theory}, volume~18.
\newblock Springer Science \& Business Media, 2012.

\bibitem[Eis88]{Eis88}
D.K. Eisen.
\newblock {\em Localized Ext groups over the Steenrod algebra}.
\newblock PhD thesis, Princeton University, 1988.

\bibitem[GJMS{\etalchar{+}}06]{LMS06}
L~Gaunce~Jr, J~Peter May, Mark Steinberger, et~al.
\newblock {\em Equivariant stable homotopy theory}, volume 1213.
\newblock Springer, 2006.

\bibitem[GM95]{GM95}
John Patrick~Campbell Greenlees and J~Peter May.
\newblock {\em Generalized {T}ate cohomology}, volume 543.
\newblock American Mathematical Soc., 1995.

\bibitem[Gre87]{Gre87}
J.~P.~C. Greenlees.
\newblock Representing {T}ate cohomology of {$G$}-spaces.
\newblock {\em Proc. Edinburgh Math. Soc. (2)}, 30(3):435--443, 1987.

\bibitem[HM97]{HM97}
Lars Hesselholt and Ib~Madsen.
\newblock On the {K}-theory of finite algebras over {W}itt vectors of perfect
  fields.
\newblock {\em Topology}, 36(1):29--101, 1997.

\bibitem[HRW22]{HRW22}
Jeremy {Hahn}, Arpon {Raksit}, and Dylan {Wilson}.
\newblock {A motivic filtration on the topological cyclic homology of
  commutative ring spectra}.
\newblock {\em arXiv e-prints}, page arXiv:2206.11208, June 2022.

\bibitem[HW22]{HW22}
Jeremy Hahn and Dylan Wilson.
\newblock Redshift and multiplication for truncated {Brown}-{Peterson} spectra.
\newblock {\em Ann. Math. (2)}, 196(3):1277--1351, 2022.

\bibitem[Jam55]{Jam55}
I.~M. James.
\newblock Reduced product spaces.
\newblock {\em Ann. of Math. (2)}, 62:170--197, 1955.

\bibitem[Jam56]{Jam56}
I.~M. James.
\newblock The suspension triad of a sphere.
\newblock {\em Ann. of Math. (2)}, 63:407--429, 1956.

\bibitem[KM24]{KM23}
Liam Keenan and Jonas McCandless.
\newblock A chromatic vanishing result for {TR}.
\newblock {\em Proc. Am. Math. Soc.}, 152(9):3705--3713, 2024.

\bibitem[KS06]{KS06}
Masaki Kashiwara and Pierre Schapira.
\newblock {\em Categories and sheaves}, volume 332 of {\em Grundlehren Math.
  Wiss.}
\newblock Berlin: Springer, 2006.

\bibitem[Law20]{Law20}
Tyler Lawson.
\newblock {{\(E_n\)}}-spectra and {Dyer}-{Lashof} operations.
\newblock In {\em Handbook of homotopy theory}, pages 793--849. Boca Raton, FL:
  CRC Press, 2020.

\bibitem[LDMA80]{LDMA80}
WH~Lin, DM~Davis, ME~Mahowald, and JF~Adams.
\newblock Calculation of {L}in's {E}xt groups.
\newblock In {\em Mathematical Proceedings of the Cambridge Philosophical
  Society}, volume~87, pages 459--469. Cambridge Univ Press, 1980.

\bibitem[Lin77]{Lin77}
Bertrand I-peng Lin.
\newblock Semiperfect coalgebras.
\newblock {\em J. Algebra}, 49(2):357--373, 1977.

\bibitem[LMMT24]{LMMT24}
Markus Land, Akhil Mathew, Lennart Meier, and Georg Tamme.
\newblock Purity in chromatically localized algebraic {{\(K\)}}-theory.
\newblock {\em J. Am. Math. Soc.}, 37(4):1011--1040, 2024.

\bibitem[LN05]{LN05}
Sverre Lun{\o}e-Nielsen.
\newblock {\em The {S}egal conjecture for topological {H}ochschild homology of
  commutative {S}-algebras}.
\newblock Phd thesis, University of Oslo, Oslo, Norway, August 2005.
\newblock Available at
  \url{https://www.duo.uio.no/bitstream/handle/10852/95537/DrScient-Lunoe-Nielsen-2005.pdf?sequence=1}.

\bibitem[LNR11]{LNR11}
Sverre Lun{\o}e-Nielsen and John Rognes.
\newblock The {S}egal conjecture for topological {H}ochschild homology of
  complex cobordism.
\newblock {\em Journal of Topology}, 4(3):591--622, 2011.

\bibitem[LNR12]{LNR12}
Sverre Lun{\o}e-Nielsen and John Rognes.
\newblock The topological {S}inger construction.
\newblock {\em Documenta Mathematica}, 17:861--909, 2012.

\bibitem[Mah79]{Mah79}
Mark Mahowald.
\newblock Ring spectra which are {T}hom complexes.
\newblock {\em Duke Mathematical Journal}, 46(3):549--559, 1979.

\bibitem[Mar11]{Mar11}
Harvey~Robert Margolis.
\newblock {\em Spectra and the {S}teenrod {A}lgebra: {M}odules over the
  {S}teenrod algebra and the stable homotopy category}.
\newblock Elsevier, 2011.

\bibitem[May65]{May65}
J.~P. May.
\newblock The cohomology of restricted {Lie} algebras and of {Hopf} algebras.
\newblock {\em Bull. Am. Math. Soc.}, 71:372--377, 1965.

\bibitem[May77]{MQRT77}
J.~Peter May.
\newblock {\em {$E_{\infty }$} ring spaces and {$E_{\infty }$} ring spectra}.
\newblock Lecture Notes in Mathematics, Vol. 577. Springer-Verlag, Berlin-New
  York, 1977.
\newblock With contributions by Frank Quinn, Nigel Ray, and J\o rgen Tornehave.

\bibitem[Mil81]{Mil81}
Haynes~R Miller.
\newblock On relations between {A}dams spectral sequences, with an application
  to the stable homotopy of a {M}oore space.
\newblock {\em Journal of Pure and Applied Algebra}, 20(3):287--312, 1981.

\bibitem[MR99]{MR99}
Mark Mahowald and Charles Rezk.
\newblock Brown-{C}omenetz duality and the {A}dams spectral sequence.
\newblock {\em Amer. J. Math.}, 121(6):1153--1177, 1999.

\bibitem[MRS01]{MRS01}
Mark Mahowald, Douglas Ravenel, and Paul Shick.
\newblock The triple loop space approach to the telescope conjecture.
\newblock In {\em Homotopy methods in algebraic topology ({B}oulder, {CO},
  1999)}, volume 271 of {\em Contemp. Math.}, pages 217--284. Amer. Math. Soc.,
  Providence, RI, 2001.

\bibitem[MS93]{MS93}
James~E McClure and RE~Staffeldt.
\newblock On the topological {H}ochschild homology of bu, {I}.
\newblock {\em American Journal of Mathematics}, 115(1):1--45, 1993.

\bibitem[NS18]{NS18}
Thomas Nikolaus and Peter Scholze.
\newblock On topological cyclic homology.
\newblock {\em Acta Mathematica}, 221(2):203--409, 2018.

\bibitem[Qui72]{Qui72}
Daniel Quillen.
\newblock On the cohomology and {K}-theory of the general linear groups over a
  finite field.
\newblock {\em Annals of Mathematics}, pages 552--586, 1972.

\bibitem[Rav86]{Rav86}
Douglas~C. Ravenel.
\newblock {\em Complex cobordism and stable homotopy groups of spheres}, volume
  121 of {\em Pure Appl. Math., Academic Press}.
\newblock Academic Press, New York, NY, 1986.

\bibitem[Rog99]{Rog99}
John Rognes.
\newblock Topological cyclic homology of the integers at two.
\newblock {\em Journal of Pure and Applied Algebra}, 134(3):219--286, 1999.

\bibitem[Rog02]{Rog02}
John Rognes.
\newblock Two-primary algebraic {$K$}-theory of pointed spaces.
\newblock {\em Topology}, 41(5):873--926, 2002.

\bibitem[Rog11]{Rog11}
John Rognes.
\newblock Introduction to redshift, September 2011.
\newblock Notes from talk at {O}berwolfach. Available at
  http://folk.uio.no/rognes/papers/red.pdf.

\bibitem[{Sal}23]{Sal23}
A.~{Salch}.
\newblock {Ravenel's May spectral sequence collapses immediately at large
  primes}.
\newblock {\em arXiv e-prints}, page arXiv:2312.17185, December 2023.

\bibitem[Was82]{Was82}
Lawrence~C. Washington.
\newblock {\em Introduction to cyclotomic fields}, volume~83 of {\em Grad.
  Texts Math.}
\newblock Springer, Cham, 1982.

\end{thebibliography}
\end{document}